\documentclass[11pt]{amsart}

\usepackage[all]{xy}
\usepackage{mathrsfs}
\usepackage{amsmath, amsthm, amssymb, amscd, amsgen,amsxtra,amsfonts, amsbsy}
\allowdisplaybreaks
\usepackage[all]{xy}
\usepackage{verbatim}

\usepackage{tikz}
\usepackage{pgfplots}
\usepackage{tikz-cd}
\usepackage{etex}
\usepackage{calligra,mathrsfs}
\usepackage{color}
\definecolor{shadecolor}{gray}{0.875}
\definecolor{dblue}{rgb}{0,0,.6}
\usepackage[colorlinks=true, linkcolor=dblue, citecolor=dblue, filecolor = dblue, menucolor = dblue, urlcolor = dblue]{hyperref}
\usepackage{mathtools}
\usepackage{extarrows}
\usepackage{ifthen}
\usepackage{bm}

\usepackage{graphicx}
\usepackage{caption,rotating}
\usepackage{color}
\usepackage{subcaption}

\usepackage{times}
\newcommand{\mathds}[1]{{\mathbb #1}}

\numberwithin{equation}{subsection}

\begin{document}
%
%   D e f i n i t i o n s
%
%
\theoremstyle{definition}
\newtheorem{Definition}{Definition}[section]
\newtheorem*{Definitionx}{Definition}
\newtheorem{Convention}{Definition}[section]
\newtheorem{Construction}{Construction}[section]
\newtheorem{Example}[Definition]{Example}
\newtheorem{Examples}[Definition]{Examples}
\newtheorem{Remark}[Definition]{Remark}
\newtheorem*{Remarkx}{Remark}
\newtheorem{Remarks}[Definition]{Remarks}
\newtheorem{Caution}[Definition]{Caution}
\newtheorem{Conjecture}[Definition]{Conjecture}
\newtheorem*{Conjecturex}{Conjecture}
\newtheorem{Question}[Definition]{Question}
\newtheorem*{Questionx}{Question}
\newtheorem*{Acknowledgements}{Acknowledgements}
\newtheorem*{Notation}{Notation}
\newtheorem*{Organization}{Organization}
\newtheorem*{Disclaimer}{Disclaimer}
\theoremstyle{plain}
\newtheorem{Theorem}[Definition]{Theorem}
\newtheorem*{Theoremx}{Theorem}

\newtheorem*{thm: regeneration}{Theorem \ref{thm: regeneration}}
\newtheorem*{thm: curves on complex k3}{Theorem \ref{thm: curves on complex k3}}
\newtheorem*{cor: curves on pos char k3}{Corollary \ref{cor: curves on pos char k3}}
\newtheorem*{def: regeneration}{Definition \ref{def: regeneration}}
\newtheorem*{prop: one reduction}{Proposition \ref{prop: one reduction}}
\newtheorem*{cor: reduction}{Proposition \ref{cor: reduction}}
\newtheorem*{thm: nodal curves}{Theorem \ref{thm: nodal curves}}
\newtheorem*{thm: curves on complex k3 picard rank r}{Theorem \ref{thm: curves on complex k3 picard rank r}}
\newtheorem*{thmA}{Theorem A}
\newtheorem*{thmB}{Theorem B}
\newtheorem*{thmC}{Theorem C}
\newtheorem*{thmD}{Theorem D}

\newtheorem{Proposition}[Definition]{Proposition}
\newtheorem*{Propositionx}{Proposition}
\newtheorem{Lemma}[Definition]{Lemma}
\newtheorem{Corollary}[Definition]{Corollary}
\newtheorem*{Corollaryx}{Corollary}
\newtheorem{Fact}[Definition]{Fact}
\newtheorem{Facts}[Definition]{Facts}
\newtheoremstyle{voiditstyle}{3pt}{3pt}{\itshape}{\parindent}
{\bfseries}{.}{ }{\thmnote{#3}}
\theoremstyle{voiditstyle}
\newtheorem*{VoidItalic}{}
\newtheoremstyle{voidromstyle}{3pt}{3pt}{\rm}{\parindent}
{\bfseries}{.}{ }{\thmnote{#3}}
\theoremstyle{voidromstyle}
\newtheorem*{VoidRoman}{}

\newenvironment{specialproof}[1][\proofname]{\noindent\textit{#1.} }{\qed\medskip}
\newcommand{\blowup}{\rule[-3mm]{0mm}{0mm}}
\newcommand{\cal}{\mathcal}
\newcommand{\Aff}{{\mathds{A}}}
\newcommand{\BB}{{\mathds{B}}}
\newcommand{\CC}{{\mathds{C}}}
\newcommand{\EE}{{\mathds{E}}}
\newcommand{\FF}{{\mathds{F}}}
\newcommand{\GG}{{\mathds{G}}}
\newcommand{\HH}{{\mathds{H}}}
\newcommand{\NN}{{\mathds{N}}}
\newcommand{\ZZ}{{\mathds{Z}}}
\newcommand{\PP}{{\mathds{P}}}
\newcommand{\QQ}{{\mathds{Q}}}
\newcommand{\RR}{{\mathds{R}}}

\newcommand{\Liea}{{\mathfrak a}}
\newcommand{\Lieb}{{\mathfrak b}}
\newcommand{\Lieg}{{\mathfrak g}}
\newcommand{\Liem}{{\mathfrak m}}

\newcommand{\ideala}{{\mathfrak a}}
\newcommand{\idealb}{{\mathfrak b}}
\newcommand{\idealg}{{\mathfrak g}}
\newcommand{\idealm}{{\mathfrak m}}
\newcommand{\idealp}{{\mathfrak p}}
\newcommand{\idealq}{{\mathfrak q}}
\newcommand{\idealI}{{\cal I}}

\newcommand{\lin}{\sim}
\newcommand{\num}{\equiv}
\newcommand{\dual}{\ast}
\newcommand{\iso}{\cong}
\newcommand{\homeo}{\approx}

\newcommand{\mm}{{\mathfrak m}}
\newcommand{\pp}{{\mathfrak p}}
\newcommand{\qq}{{\mathfrak q}}
\newcommand{\rr}{{\mathfrak r}}
\newcommand{\pP}{{\mathfrak P}}
\newcommand{\qQ}{{\mathfrak Q}}
\newcommand{\rR}{{\mathfrak R}}

\newcommand{\OO}{{\cal O}}
\newcommand{\numero}{{n$^{\rm o}\:$}}
\newcommand{\mf}[1]{\mathfrak{#1}}
\newcommand{\mc}[1]{\mathcal{#1}}

\newcommand{\into}{{\hookrightarrow}}
\newcommand{\onto}{{\twoheadrightarrow}}
\newcommand{\Spec}{{\rm Spec}\:}
\newcommand{\BigSpec}{{\rm\bf Spec}\:}
\newcommand{\Spf}{{\rm Spf}\:}
\newcommand{\Proj}{{\rm Proj}\:}
\newcommand{\Pic}{{\rm Pic }}
\newcommand{\MW}{{\rm MW }}
\newcommand{\Br}{{\rm Br}}
\newcommand{\NS}{{\rm NS}}

\newcommand{\Sym}{{\mathfrak S}}

\newcommand{\Aut}{{\rm Aut}}
\newcommand{\Autp}{{\rm Aut}^p}
\newcommand{\ord}{{\rm ord}}
\newcommand{\coker}{{\rm coker}\,}

\newcommand{\divisor}{{\rm div}}
\newcommand{\Def}{{\rm Def}}
\newcommand{\rank}{\mathop{\mathrm{rank}}\nolimits}
\newcommand{\Ext}{\mathop{\mathrm{Ext}}\nolimits}
\newcommand{\EXT}{\mathop{\mathscr{E}{\kern -2pt {xt}}}\nolimits}
\newcommand{\Hom}{\mathop{\mathrm{Hom}}\nolimits}
\newcommand{\HOM}{\mathop{\mathscr{H}{\kern -3pt {om}}}\nolimits}

\newcommand{\calA}{\mathscr{A}}
\newcommand{\calH}{\mathscr{H}}
\newcommand{\calL}{\mathscr{L}}
\newcommand{\calM}{\mathscr{M}}
\newcommand{\calN}{\mathscr{N}}
\newcommand{\calX}{\mathscr{X}}
\newcommand{\calK}{\mathscr{K}}
\newcommand{\calD}{\mathscr{D}}
\newcommand{\calY}{\mathscr{Y}}

\newcommand{\piet}{{\pi_1^{\rm \acute{e}t}}}
\newcommand{\Het}[1]{{H_{\rm \acute{e}t}^{{#1}}}}
\newcommand{\Hfl}[1]{{H_{\rm fl}^{{#1}}}}
\newcommand{\Hcris}[1]{{H_{\rm cris}^{{#1}}}}
\newcommand{\HdR}[1]{{H_{\rm dR}^{{#1}}}}
\newcommand{\hdR}[1]{{h_{\rm dR}^{{#1}}}}

\newcommand{\loc}{{\rm loc}}
\newcommand{\et}{{\rm \acute{e}t}}

\newcommand{\defin}[1]{{\bf #1}}

\newcommand{\blue}{\textcolor{blue}}
\newcommand{\red}{\textcolor{red}}

\renewcommand{\HH}{{\rm{H}}}

\title{Rational curves on lattice-polarised K3 surfaces}

\author{Xi Chen}
\address{632 Central Academic Building, University of Alberta, Edmonton, Alberta T6G 2G1, Canada}
\email{xichen@math.ualberta.ca}

\author{Frank Gounelas}
\address{TU M\"unchen, Zentrum Mathematik - M11, Boltzmannstr. 3, 85748 Garching bei M\"unchen, Germany}
\email{gounelas@ma.tum.de}

\author{Christian Liedtke}
\email{liedtke@ma.tum.de}

\date{\today}

\begin{abstract}
Fix a K3 lattice $\Lambda$ of rank two and $L\in\Lambda$ a big and nef divisor that is positive enough. We prove that the
generic $\Lambda$-polarised K3 surface has an integral nodal rational curve in the linear system $|L|$, in
particular strengthening previous work of the first named author. The technique is by degeneration, and also works for 
many lattices of higher rank.
\end{abstract}

\maketitle

\section{Introduction}

In \cite{chen} the first named author proved the existence of integral nodal rational curves in $|nL|$ on a generic K3
surface with Picard group generated by $L$ for all $n>0$ and $L^2\geq4$. In this paper, we follow a similar strategy and
prove the following (see Section \ref{sec: rigidifiers} for a more precise statement).

\begin{thmA}
    Let $a, b\in \ZZ$ and $d\in \ZZ_{>0}$ satisfying $4bd-a^2<0$, and let $\Lambda$ be a lattice of rank two with
    intersection matrix
    \begin{equation*}
    \begin{bmatrix}
    2d & a\\
    a & 2b
    \end{bmatrix}_{2\times2}.
    \end{equation*}
    Then for any ample class $L\in \Lambda$ which is the sum of three ample classes, there exists a Zariski-open dense
    subset $U_L$ of the moduli space of $\Lambda$-polarised K3 surfaces $M_\Lambda$, such that there is an integral
    nodal rational curve in $|L|$ for every K3 surface $X\in U_L$.
\end{thmA}

The method of proof of Theorem A is a sequence of two degenerations. One proceeds first by degenerating to a smooth K3
surface of higher Picard rank that contains several $(-2)$-curves. The main technical difficulty is that one must
distinguish between rank two lattices of even and odd discriminant, and each case requires a different degeneration,
which significantly adds to the length of the argument. In particular, we prove that any rank two K3 lattice embeds
primitively into one of the following two lattices 
\begin{equation}
\label{twomatrices}
\begin{bmatrix}
2\\
& -2\\
&& -2\\
&&& \ddots\\
&&&& -2
\end{bmatrix}
\hspace{15pt}\text{and}\hspace{15pt} 
\begin{bmatrix}
0 & 1\\
1 & -2\\
&& -2\\
&&& \ddots\\
&&&& -2
\end{bmatrix}
\end{equation}
of size $r_1\times r_1$ and $r_2\times r_2$ for some $r_1\leq7$ and $r_2\leq 5$ respectively. 

To prove existence of integral nodal rational curves in $|L|$ for $L$ a big and nef class in one of the above two
matrices we degenerate further, like in \cite{chen}, to unions of smooth rational surfaces \`a la
Ciliberto--Lopez--Miranda. One constructs now reducible, \textit{limiting rational curves} with prescribed singularities
in such log K3 surfaces, which now deform out to integral and nodal rational curves on the general K3 with lattice as
above.

The technique in fact works to produce many nodal rational curves for a general K3 with Picard lattice which embeds into
one of \eqref{twomatrices}. At the end of this paper, we show how this can work for the two special rank four lattices
of Nikulin \cite{nikulin}, for which the K3 in question has finite automorphism group and is not elliptic.

The above results will be used in a follow up paper \cite{regenerationinfinite} whose main result completes the project
(initiated by Bogomolov--Mumford) of showing that every complex projective K3 surface contains infinitely many rational
curves.

\begin{Notation}
A \textit{K3 surface $X$} will be a geometrically integral, smooth, proper and
separated scheme of relative dimension $2$ over the complex numbers, so that $\omega_X\cong\OO_X$ and $\HH^1(X,\OO_X)=0$.
Let $S$ be a connected base scheme. Then a morphism
$$
   \begin{tikzcd}f:\calX \ar{r} &S\end{tikzcd}
$$
is a \textit{smooth family of surfaces} if $f$ is a smooth and proper morphism of algebraic spaces of relative dimension
two whose geometric fibres are irreducible. In particular a \textit{family of K3 surfaces} is a family of surfaces where every
fibre is a K3 surface as above.
A property that holds for a \textit{general point} in a set will mean that it holds for all points of a Zariski-open subset,
whereas a \textit{very general point} will be one in the complement of countably many Zariski-closed subsets.
\end{Notation}

\begin{Acknowledgements}
We thank D. Huybrechts, K. Ito, M. Kemeny, G. Martin and J. C. Ottem for discussions and comments and in particular A.
Knutsen for remarks and corrections. The first named author is partially supported by the NSERC Discovery Grant 262265.
The second and third named authors are supported by the ERC Consolidator Grant 681838 ``K3CRYSTAL''.
\end{Acknowledgements}

\section{Degenerations of type II}
\label{sec: type II degeneration}

In this section, we discuss degenerations of K3 surfaces that are of type II in the sense of Kulikov \cite{kulikov}, see
also \cite{persson, pp}. We will need these degenerations in order to produce nodal rational curves in
the next section. 

\begin{Definition}
A {\em degeneration of type II (in the sense of Kulikov)} of K3 surfaces is a flat and proper
family $\pi: {\calX}\to S$, where $S$ is the spectrum of a DVR with residue field $\CC$ and
where the geometric
generic fiber $\overline{{\calX}}_\eta$ is a K3 surface and the special fiber ${\cal X}_0$
is a union $Y_1\cup Y_2\cup \ldots \cup Y_m$, such that
\begin{enumerate}
\item each $Y_i$ is a smooth surface for all $i$ ,
\item $Y_i\cap Y_j = \emptyset$ for $|i-j|\ne 0,1$ and $Y_i$ and $Y_{i+1}$ meet transversally
  along a smooth elliptic curve $D_i\cong D$,
\item $E_1$ and $E_{m-1}$ are anti-canonical divisors in $Y_1$ and $Y_m$,
\item each $Y_i$, $1<i<m$ is a ruled surface over $D$.
\end{enumerate}
\end{Definition}
It follows from (3) that $Y_1$ and $Y_m$ are rational surfaces.
We note that the chain $Y_2\cup \ldots\cup Y_{m-1}$ of ruled surfaces can be contracted
to a family ${\calX}'\to S$ that has the same generic fiber, but the special
fiber has the form ${\calX}'_0 := Y'_1\cup Y'_2$, where $Y_1'$ and $Y_2'$ are smooth
rational surfaces meeting transversally along a smooth anti-canonical curve $D$.
In this case, $(Y_i',D)$, $i=1,2$ are two genuine log K3 surfaces.

Conversely, given a union of $Y_1\cup Y_2$ of two smooth rational surfaces meeting transversely along a
smooth anti-canonical curve $E$, one may ask whether it can be deformed to a K3 surface.
We refer the interested reader to \cite{CLM}, where this question was studied.

More generally, let $Y = Y_1\cup Y_2$ be the union of two smooth projective varieties meeting
transversely along a smooth hypersurface $D$ in each $Y_i$.
That is, \'etale locally around $D$, the union is given by $xy = 0$.
In particular, $Y$ is reducible.
The Picard group $\Pic(Y)$
is given by the exact sequence
$$
\begin{tikzcd}
 0\ar{r} & \Pic(D) \ar{r}{\imath_1^* - \imath_2^*} & \Pic(Y_1)\oplus\Pic(Y_2)\ar{r}
 & \Pic(Y)\ar{r} & 0
\end{tikzcd}
$$
where $\imath_i: D\hookrightarrow Y_i$ denotes the embedding of $D$ into $Y_i$ for $i=1,2$.
In other words, an invertible sheaf ${\calL}$ on $Y$
is given by a pair of invertible sheaves ${\calL}_i\in \Pic(Y_i)$ satisfying
\begin{equation}
 \label{descent}
  \imath_1^* \calL_1 \,\cong\, \imath_2^*\calL_2.
\end{equation}
(For the reader that wants to avoid this descent construction, we have an alternative realization of $Y$ as below.)
In particular, $Y$ is projective if and only if there exists an ample invertible sheaf ${\calL}\in \Pic(Y)$ if and only if there exists a pair of ample invertible sheaves ${\calL}_i$ on $Y_i$ that  satisfies \eqref{descent}. If such a $\calL$ exists, then we can embed $Y$ into some $\PP^n$ via $|\calL^{\otimes m}|$ for $m$ sufficiently large.

Alternatively, we may start with two smooth projective varieties $Y_1, Y_2$, embeddings $\imath_i: D\hookrightarrow Y_i$ of a smooth hypersurface
$D$ in $Y_i$ and two ample line bundles ${\calL}_i\in \Pic(Y_i)$ satisfying \eqref{descent} for $i=1,2$.
Let us choose $m$ sufficiently large such that ${\calL}_i^{\otimes m}$ are very ample
and the maps $\HH^0(Y_i, \calL_i^{\otimes m})\to \HH^0(D,\calL_i^{\otimes m})$ are surjective
for $i=1,2$. We choose a basis $\{(s_{1j}, s_{2j}): j=0,1,\ldots,n\}$ for the kernel of the map
$$
\begin{tikzcd}
\HH^0(Y_1, \calL_1) \oplus \HH^0(Y_2,\calL_2) \ar{r}{\imath_1^*-\imath_2^*} & \HH^0(D, \imath_1^* \calL_1) \ar[equal]{r}
& \HH^0(D, \imath_2^* \calL_2)
\end{tikzcd}
$$
and define the maps $\phi_i: Y_i\to \PP^n$ by $(s_{i0},s_{i1},\ldots,s_{in})$ for $i=1,2$. Then we see that $Y =
\phi_1(Y_1)\cup \phi_2(Y_2)$ is the union of two smooth projective varieties meeting along $\phi_1(D) = \phi_2(D)$ such
that $\phi_{i,*} T_{D,p} = \phi_{1,*}T_{Y_1,p}\cap \phi_{2,*} T_{Y_2,p}$ for the tangent spaces of $Y_i$ and $D$ at
$p\in D$, as subspaces of $T_{\PP^n, \phi(p)}$.

The
first order embedded deformations of $Y\subset P := \PP^n$ are classified by $\HH^0(Y,{\calN}_Y)$, where
${\calN}_Y$
denotes the normal sheaf of $Y\subset P$.  We note that these deformations of $Y$ preserve the line bundle
$\calL^{\otimes m}$. We want to deform $Y$ to a
smooth variety, that is, to ``smooth'' out $D$, which is the singular locus of $Y$. This is governed by the map
\begin{equation}
 \label{K3RATNOTESE002}
 \begin{tikzcd}
 \HH^0(Y,\calN_Y) \ar{r} & \HH^0(Y,T_Y^1),
 \end{tikzcd}
\end{equation}
where $T^1_Y$ is the sheaf of $\OO_Y$-modules
\begin{equation}\label{K3RATNOTESE003}
T_Y^1 \,:=\, \EXT(\Omega_Y, \OO_Y) \,\cong\, {\calN}_{D/Y_1}\otimes {\calN}_{D/Y_2},
\end{equation}
where ${\calN}_{D/Y_i}$ denotes the normal bundle of $D$ in $Y_i$.
A general deformation
of $Y\subset P$ smooths out $D$ if the embedded deformations of $Y\subset P$ are unobstructed and if
the image of the map \eqref{K3RATNOTESE002} is base point free on $D$.

Let $\calX\subset P\times \Delta$ be a deformation of $Y$ given by some $\xi\in \HH^0(Y,\calN_Y)$
with $\calX_0 = Y$,
where $\Delta = \{|t| < 1\}$ is the unit disk.
Then $\calX$ is singular along the vanishing locus $z(\rho(\xi))$ of $\rho(\xi)$, where $\rho$ is the map \eqref{K3RATNOTESE002}. If $z(\rho(\xi))$ is smooth as a closed subscheme of $D$, then $\calX$ has singularities of type
$$
\CC[[x_1,x_2,\ldots,x_n,t]]/(x_1x_2 - tx_3)
$$
at $z(\rho(\xi))$ and hence the generic fiber $\calX_\eta$ is smooth.
So a general deformation of $Y\subset P$ smooths out $D$ under the above hypotheses.

A general deformation of $Y\subset P$, a priori, only preserves $\calL^{\otimes m}$.
For the above family $\calX$, the restriction of $\calH = \OO_{\calX}(1)$ to $Y$ is obviously $\calL^{\otimes m}$. On the other hand, in our application, $\HH^2(Y, \ZZ)$ is always
torsion free; by Mayer-Vietoris, this is guaranteed if $\HH^1(Y_i) = 0$ and $\HH^1(D,\ZZ)$ and $\HH^2(Y_i,\ZZ)$ are torsion free. By deformation retraction, $\HH^2(\calX, \ZZ) \cong \HH^2(Y,\ZZ)$ is torsion free. Consequently, $(1/m) c_1(\calH)\in \HH^2(\calX, \ZZ)$ and $\calL$ extends to a line bundle on $\calX$.
In conclusion, a general deformation of $Y$ preserves $\calL$ if $\HH^2(Y, \ZZ)$ is torsion free.

Moreover, if we construct $Y$ with arbitrary Picard rank $r:=\rho(Y)$, we can deform $Y$ to preserve $\Pic(Y)$ as follows. We choose very ample line bundles
${\calH}_1,\ldots,{\calH}_r$ on $Y$ which generate $\Pic_\QQ(Y)$ and embed $Y$ into
$P = \PP^{n_1}\times \PP^{n_2}\times \ldots\times \PP^{n_r}$ via the complete linear systems $|{\calH}_i|$.
If the embedded deformations of $Y\subset P$ smooth out $D$,
we can deform $Y$ to a smooth variety while preserving $\Pic_\QQ(Y)$.
In addition, as commented above, if $\HH^2(Y,\ZZ)$ is torsion-free,
then $\Pic(Y)$ is preserved when $Y$ deforms in $P$.
Using the techniques in \cite[Section 1]{CLM}, we can prove the following theorem on the deformation of $Y\subset P$.

\begin{Theorem}[Ciliberto--Lopez--Miranda + $\varepsilon$]
\label{K3RATNOTESTHMCLM}
Let $Y = Y_1\cup Y_2$ be a union of two smooth projective varieties meeting transversely along a smooth hypersurface $D$
in $Y_i$. Suppose that $Y$ is embedded into $P = \PP^{n_1}\times \PP^{n_2}\times \ldots\times \PP^{n_r}$ by very ample line bundles ${\calH}_1,\ldots,{\calH}_r\in \Pic(Y)$ satisfying that
${\calH}_1,\ldots,{\calH}_r$ are linearly independent in $\HH^{1,1}(Y_1)$.
\begin{enumerate}
\item
$T_Y^1$ is isomorphic to the cokernel of the inclusion $\calN_{Y_i} \to \calN_Y\otimes \OO_{Y_i}$
for $i=1,2$, that is, the sequence
\begin{equation}\label{K3RATNOTESE004}
\begin{tikzcd}
0 \ar{r} & \calN_{Y_i} \ar{r} & \calN_Y\Big|_{Y_i} \ar{r} & T_Y^1
\ar{r} & 0
\end{tikzcd}
\end{equation}
is exact for $i=1,2$.
\item
$\HH^1(Y,\calN_Y) = 0$ and the map \eqref{K3RATNOTESE002} is surjective if
\begin{equation}\label{K3RATNOTESE005}
 \begin{array}{lclcl}
  \HH^1(Y_i, \calH_j)     &=& \HH^1(Y_1, \calH_j(-D)) &=& 0\\
 \HH^1(\calN_{D/Y_2}) &=& \HH^1(\calN_{D/Y_1} \otimes \calN_{D/Y_2})  &=& 0\\
  \HH^2(\OO_{Y_i}) &=& \HH^2(T_{Y_i}) \,=\, \HH^2(T_{Y_1}(-D)) &=& 0
 \end{array}
\end{equation}
and moreover
\begin{equation}\label{K3RATNOTESE006}
\text{either } \HH^2(\OO_{Y_1}(-D)) = 0
\text{ or } K_{Y_1} + D = 0 \text{ and }
\dim Y = 2
\end{equation}
 for $i=1,2$ and $j=1,2,\ldots,r$.
\end{enumerate}
\end{Theorem}

\begin{proof}
We basically follow the same argument in \cite{CLM}: \eqref{K3RATNOTESE004} is a consequence of the commutative diagram
$$
\begin{tikzcd}
T_P\Big|_{Y_i}\ar[equal]{d} \ar{r} & \calN_{Y_i} \ar{d} \ar{r} & 0\\
T_P\Big|_{Y_i} \ar{r} & \calN_Y \Big|_{Y_i} \ar{r} & T_Y^1\ar{r} & 0
\end{tikzcd}
$$
which shows that $\coker(\calN_{Y_i} \to \calN_Y\otimes \OO_{Y_i})$ surjects onto $T_Y^1$. And since
$\coker(\calN_{Y_i} \to \calN_Y\otimes \OO_{Y_i})$ and
$T_Y^1$ are line bundles supported on $D$, the surjection must be an isomorphism and we obtain \eqref{K3RATNOTESE004}.

To prove $\HH^1(Y,\calN_Y) = 0$ and the surjectivity of \eqref{K3RATNOTESE002}, we combine \eqref{K3RATNOTESE004} with the exact sequence
$$
\begin{tikzcd}
0 \ar{r} & \calN_{Y}\otimes \OO_{Y_1}(-D) \ar{r} & \calN_Y \ar{r} & \calN_Y\otimes \OO_{Y_2} \ar{r} & 0.
\end{tikzcd}
$$
With these two exact sequences, it suffices to prove
$$
\HH^1(\calN_{Y_1}(-D)) = \HH^1(\calN_{Y_1})
= \HH^1(\calN_{Y_2}) = \HH^1(\calN_{D/Y_2}) = \HH^1(T_Y^1) = 0.
$$
The vanishing of these cohomological groups mostly follows from \eqref{K3RATNOTESE005}. Let us say something about $\HH^1(\calN_{Y_1}(-D)) = 0$.

If $\HH^2(\OO_{Y_1}(-D)) = 0$, then the vanishing of $\HH^1(\calN_{Y_1}(-D))$ follows from \eqref{K3RATNOTESE005} and the exact sequences
$$
\begin{tikzcd}
0 \ar{r} & T_{Y_1}\ar{r} & T_P\otimes \OO_{Y_1} \ar{r} & \calN_{Y_1} \ar{r} & 0\\
0 \ar{r} & \OO_{Y_1}^{\oplus r} \ar{r}
& \displaystyle{\sum_{i=1}^r \calH_i^{\oplus (n_i+1)}} \ar{r} & T_P\otimes \OO_{Y_1} \ar{r} & 0.
\end{tikzcd}
$$
If we instead have $K_{Y_1} + D = 0$ and
$\dim Y = 2$ in \eqref{K3RATNOTESE006}, then the vanishing of $\HH^1(\calN_{Y_1}(-D))$ follows from \eqref{K3RATNOTESE005} and the injectivity of the map
$$
\begin{tikzcd}
\HH^1(\Omega_P) \ar[hook]{r} \ar[equal]{d} & \HH^1(\Omega_{Y_1})
\ar[equal]{d}\\
\HH^1(T_P\otimes K_{Y_1})^\vee \ar[equal]{d} \ar{r} & \HH^1(T_{Y_1}\otimes K_{Y_1})^\vee\ar[equal]{d}\\
\HH^1(T_P\otimes \OO_{Y_1}(-D))^\vee \ar{r} & \HH^1(T_{Y_1}(-D))^\vee
\end{tikzcd}
$$
where the injectivity of $\HH^1(\Omega_P) \to \HH^1(\Omega_{Y_1})$ is a consequence of the hypothesis that ${\calH}_1,\ldots,{\calH}_r$ are linearly independent in $\HH^{1,1}(Y_1)$.
\end{proof}

If $Y=Y_1\cup Y_2$ is a degeneration of type II of K3 surfaces with $Y_i$ Fano varieties (that is, del~Pezzo surfaces) for $i=1,2$,
then the hypotheses \eqref{K3RATNOTESE005} and \eqref{K3RATNOTESE006} are clearly satisfied.
In this case, $Y$ can be deformed to a smooth projective K3 surface of Picard rank $r = \rho(Y)$.
Thus, we conclude the following.

\begin{Theorem}
\label{K3RATNOTESTHM000}
Let $Y = Y_1\cup Y_2$ be the union of two smooth rational surfaces $Y_i$
meeting transversely along a smooth anti-canonical curve $D$ in $Y_i$ for $i=1,2$ satisfying that
$$
\begin{aligned}
\deg \calN_{D/Y_1} + \deg\calN_{D/Y_2} &= K_{Y_1}^2 + K_{Y_2}^2 \ge 1\hspace{12pt}\text{and}\\
\deg\calN_{D/Y_2} &= K_{Y_2}^2 \ge 1.
\end{aligned}
$$
Suppose that there exist $\calL_1, \calL_2, \ldots, \calL_r\in \Pic(Y)$ such that
$\calL_i$ are linearly independent in $\HH^2(Y_1)$ and the subgroup of $\Pic(Y)$ generated by $\calL_i$ contains an ample line bundle on $Y$.
Then $Y$ can be deformed to a projective K3 surface of Picard rank $r$.
More precisely, there exists a flat projective family
$\pi: \calX\to \Spec \CC[[t]]$ such that its central fiber is $\calX_0 = Y$,
its generic fiber $\calX_\eta$ is a projective K3 surface of Picard rank $r$ and the image of $\Pic(\calX)\to \Pic(Y)$ contains $\calL_1, \calL_2, \ldots, \calL_r$.
\end{Theorem}

\begin{proof}
Let us choose sufficiently ample $\calH_1, \calH_2, \ldots, \calH_r\in \Pic(Y)$ such that
$$
\text{Span}_\QQ \{ \calL_1, \calL_2, \ldots, \calL_r \} = \text{Span}_\QQ \{ \calH_1, \calH_2, \ldots, \calH_r \}
$$
in $\Pic(Y)$. We embed $Y$ into $P = \PP^{n_1}\times \PP^{n_2}\times \ldots \times \PP^{n_r}$ by $|\calH_i|$.
By Theorem \ref{K3RATNOTESTHMCLM}, the embedded deformations of $Y\subset P$ are unobstructed and a general deformation
$Y'\subset P$ of $Y$ smooths out $D$; hence $Y'$ is a projective K3 surface of Picard rank at least $r$. To see that
$\rho(Y') = r$, we compute
$h^0(\calN_Y) = 20 - r + h^0(T_P)$
and conclude that the image of $\HH^0(\calN_Y)\to \text{Ext}(\Omega_Y,\OO_Y)$ has dimension $\ge 20-r$.

Let $\calX \subset P\times \CC[[t]]$ be the family given by a general deformation of $Y\subset P$. Then for every
$\calL_i$, $m_i\calL_i$ lies in the image of $\Pic(\calX)\to \Pic(Y)$ for some integer $m_i\ne 0$; since
$\HH^2(Y,\ZZ)$ is torsion free, $\calL_i$ lies in the image of $\Pic(\calX)\to \Pic(Y)$ for $i=1,2,\ldots,r$.
\end{proof}

Let $\calX$ be the family given by a general deformation of $Y$ in the above theorem. Then $\calX$ has rational double points at
$x_1,x_2,\ldots,x_s\in D$ for $s = K_{Y_1}^2 + K_{Y_2}^2$,
which are the vanishing locus of a section in $\HH^0(T_Y^1)$. Clearly, $x_i$ satisfy
\begin{equation}\label{K3RATNOTESE000}
\OO_D(x_1 + x_2 + \ldots + x_s) = {\calN}_{D/Y_1}\otimes {\calN}_{D/Y_2} = \OO_D(-K_{Y_1} - K_{Y_2}).
\end{equation}
For a general deformation of $Y$, $x_1,x_2,\ldots,x_s$ are $s$ general points with the only relation \eqref{K3RATNOTESE000} on $D$.

Even if $Y$ is not projective, we can still deform $Y$ to a K3 surface, although the resulting family is obviously non-projective. The issue of projectivity is purely technical.

\begin{Theorem}
\label{thm: k3 type2 degeneration}
Under the same hypotheses of Theorem \ref{K3RATNOTESTHM000}, except instead of assuming that the subgroup of $\Pic(Y)$
generated by the $\calL_i$ contains an ample line bundle
and $K_{Y_1}^2 + K_{Y_2}^2\ge 1$, we assume that $K_{Y_1}^2 + K_{Y_2}^2 \ge 2$.
Then
$Y$ can be deformed to a projective K3 surface of Picard rank $r+1$. More precisely, there exists a flat proper (possibly non-projective) family
$\pi: \calX\to \Spec \CC[[t]]$ such that its central fiber is $\calX_0 = Y$,
its generic fiber $\calX_\eta$ is a projective K3 surface of Picard rank $r+1$ and the image of $\Pic(\calX)\to \Pic(Y)$ contains $\calL_1, \calL_2, \ldots, \calL_r$.
\end{Theorem}

\begin{proof}
We choose an ample line bundle $M_1$ on $Y_1$ and an ample line bundle $M_2$ on $Y_2$. Let $m_i = M_iD$ on $Y_i$ for $i=1,2$ and $m = \gcd(m_1,m_2)$. Let $a_1$ and $a_2$ be positive integers such that $a_1m_1 - a_2m_2 = m$ and let $p$ be a point on $D$ such that
$$
\OO_D(a_1 M_1) = \OO_D(a_2 M_2) \otimes \OO_D(mp).
$$
Let us choose $M_i$ and $p$ such that $\OO_D(2p) \ne \OO_D(-K_{Y_1}^2 - K_{Y_2}^2)$.
Let $\widehat{Y}_1$ be the blowup of $Y_1$ at $p$. We can construct a union of $\widehat{Y} = \widehat{Y}_1\cup Y_2$ meeting transversely along $D$ such that
there is a morphism $\varphi: \widehat{Y} \to Y$ with $\varphi\big|_{\widehat{Y}_1}$
being the blowup map $\widehat{Y}_1\to Y_1$ and
$\varphi\big|_{Y_2} = \text{id}$.

For $a_1$ sufficiently large, $a_1 \varphi^* M_1 - m E$ is ample on $\widehat{Y}_1$, where $E$ is the exceptional divisor of $\widehat{Y}_1\to Y_1$. Therefore, there exists an ample line bundle $\widehat{M}$ on $\widehat{Y}$ whose restriction to $\widehat{Y}_1$ is $a_1 \varphi^* M_1 - m E$ and whose restriction to $Y_2$ is $a_2 \varphi^* M_2$.

Applying Theorem \ref{K3RATNOTESTHM000} to $\widehat{Y}$ with $\widehat{M}, \varphi^* \calL_1, \varphi^* \calL_2, \ldots, \varphi^* \calL_r$,
we obtain a flat projective family $\calY\to \Spec \CC[[t]]$
such that $\calY_0 = \widehat{Y}$, $\calY_\eta$ is a K3 surface of Picard rank $r+1$
and the image of $\Pic(\calY)\to \Pic(\widehat{Y})$ contains $\widehat{M}, \varphi^* \calL_1, \varphi^* \calL_2, \ldots, \varphi^* \calL_r$.

For a general choice of $\calY$, it has at worst rational double points on $D$ satisfying \eqref{K3RATNOTESE000}. We may assume that $\calY$ is smooth along $E$. As a complex manifold, $\calY$ admits a small contraction of $E$. Let us still use $\varphi$ to denote this map:
$$
\begin{tikzcd}
\calY \ar{r}{\varphi} \ar{d} & \calX\ar{dl}\\
\Spec \CC[[t]].
\end{tikzcd}
$$
Clearly, $\calX_0 = Y$ and $\calX_\eta$ is a projective K3 surface of Picard rank $r+1$, while $\calX$ is flat and proper but possibly non-projective over $\Spec \CC[[t]]$. The image of $\Pic(\calX)\to \Pic(Y)$ contains $\calL_1, \calL_2, \ldots, \calL_r$.
\end{proof}

In this paper, we mainly use this degeneration to construct rational curves on generic K3 surfaces. Let $\pi: \calX\to \Spec \CC[[t]]$ be the family constructed in Theorem \ref{K3RATNOTESTHM000}
and \ref{thm: k3 type2 degeneration}. Suppose that $\calX$ has rational double points $x_1,x_2,\ldots,x_s$ on $D$ satisfying \eqref{K3RATNOTESE000}. We are going to find integral (nodal) rational curves in $|\calL|$ on
$\calX_\eta$ for some $\calL\in \Pic(\calX)$. In order to do that, we construct some $\Gamma\in |\calL|$ on $\calX_0 = Y$, which we call ``limiting rational curves'', and show that $\Gamma$ can be deformed to an integral (nodal) rational curve on $\calX_\eta$.

We consider $\Gamma = f_* C$ as the image of a stable map $f: C\to Y$. Instead of deforming $\Gamma$, we try to deform the map $f$. Actually, we can construct ``deformable'' stable maps $f: C\to Y$ of arbitrary genus $g$, up to the arithmetic genus of $\calL$, as follows:
\begin{itemize}
\item $h^0(\calX_\eta, \calL) = h^0(Y, \calL)$, where $\HH^0(Y, \calL)$ is the kernel of the map
$$
\begin{tikzcd}[column sep=66pt]
\HH^0(Y_1,\calL_1) \oplus \HH^0(Y_2,\calL_2) \ar{r}{(\gamma_1,\gamma_2) \to \gamma_1 - \gamma_2}
& \HH^0(D, \calL_1)\ar[equal]{d}\\
 & \HH^0(D,\calL_2)
\end{tikzcd}
$$
where $\calL_i$ is the restriction of $\calL$ on $Y_i$ for $i=1,2$.
\item $D\not\subset \Gamma\in \PP \HH^0(Y,\calL)$.
\item $f$ maps each irreducible component $G\subset C$ birationally onto its image:
\begin{equation}\label{K3RATNOTESE025}
f_* G = f(G) \text{ for all irreducible components } G\subset C.
\end{equation}
\item Let $C_\times$ be the points of $C$ lying on two distinct components of $C$. Then
\begin{equation}\label{K3RATNOTESE016}
\begin{aligned}
&\text{for each } p\in G_1\cap G_2\subset C_\times,\ f(p)\in D\backslash \{x_1,x_2,\ldots,x_s\},\\
&\hspace{12pt}
f(G_1)\subset Y_1,\ f(G_2)\subset Y_2,\\
&\hspace{12pt}
f_{G_1} = f\Big|_{G_1}: G_1 \to Y_1,\ f_{G_2} = f\Big|_{G_2}: G_2\to Y_2\\
&\hspace{12pt}
\text{ and } v_p(f_{G_1}^* D) = v_p(f_{G_2}^* D)
\end{aligned}
\end{equation}
where $G_1$ and $G_2$ are two irreducible components of $C$ meeting at $p$
and $v_p(f_{G_i}^* D)$ is the multiplicity of $p$ in $f_{G_i}^* D$.
\item Outside of $f(C_\times)$, $\Gamma$ and $D$ only meet at $x_1,x_2,\ldots,x_s$. More precisely,
\begin{equation}\label{K3RATNOTESE020}
\begin{aligned}
&\text{for each } q\in f^{-1}(D)\backslash C_\times,\ f(q)\in \{x_1,x_2,\ldots,x_s\},\\
&\hspace{12pt} f_G = f\Big|_G: G\to Y_i \text{ and } v_q(f_G^*D) = 1
\end{aligned}
\end{equation}
where $G$ is the irreducible component of $C$ containing $q$.
\end{itemize}

Using the deformation theory of curves on ${\calX}$ as explained in \cite{chen}, we can deform $f$ to the generic fiber $\calX_\eta$. The above statement includes several improvements over \cite{chen}. For example,
we do not assume that $f_* C_i$ has simple tangency with $D$ in \eqref{K3RATNOTESE016}. The difficulties caused by loosening these restrictions on $\Gamma$ can be overcome by studying the deformation of the stable map $f: C\to \calX$ instead of the deformation of $\Gamma\subset \calX$, which is carried out in the same way as in the case that $\calX$ is a smooth family of K3 surfaces. On the other hand, these assumptions do not guarantee that $\Gamma$ can be deformed to a nodal curve on $\calX_\eta$; for that to happen, we do need the same restrictions on $\Gamma$ as in \cite{chen}.

We will make one more improvement over \cite{chen}. Instead of only considering $|\calL|$, we will also consider the
``twisted'' linear series $|\calL + m Y_1|$ on $\calX$. Note that
$\calX$ is smooth outside of $x_j$ so $Y_1$ is a Cartier divisor on
$\calX^\circ = \calX\backslash \{x_j\}$; $|\calL+mY_1|$ is interpreted as $\PP \HH^0(\calX^\circ, \calL+mY_1)$. The restrictions of $\calL+mY_1$ to $Y_i$ are
$$
(\calL+mY_1)\Big|_{Y_1} = \calL_1 - mD \text{ and } (\calL+mY_1)\Big|_{Y_2} = \calL_2 + mD
$$
respectively, where $\calL_i$ are the restrictions of $\calL$ to $Y_i$ for $i=1,2$.
Although $m$ can be chosen to be an arbitrary integer, we take $m\ge 0$ for simplicity.

The restriction of
$\gamma \in \HH^0(\calL+mY_1)$ to $Y$ consists of $\gamma_1\in \HH^0(\calL_1 - mD)$ on $Y_1$ and
$\gamma_2\in \HH^0(\calL_2 + mD)$ on $Y_2$.
Furthermore, the image of the restriction
$$
\begin{tikzcd}
\HH^0(\calX^\circ, \calL+mY_1) \ar{r} & \HH^0(Y_2\backslash \{x_i\}, \calL_2 + mD) \ar[equal]{r} & \HH^0(Y_2, \calL_2 + mD)
\end{tikzcd}
$$
is actually contained in the subspace
$$
\HH^0(\OO_{Y_2}(\calL_2 + mD) \otimes \OO_{Y_2}(-mx_1 - mx_2 - \ldots - mx_s))
$$
where $\OO_{Y_2}(-x_j)$ is the ideal sheaf of the point $x_j$ and $\OO_{Y_2}(-mx_j)$ is the $m$-th symmetric product of $\OO_{Y_2}(-x_j)$ for $j=1,2,\ldots,s$. That is,
$$
\gamma_2 \in \HH^0(\OO_{Y_2}(\calL_2 + mD - m \sum x_j)).
$$
This is easy to see after we resolve the double points of $\calX$ by blowing it up along $Y_2$.
In summary, the restriction
of $\HH^0(\calL+mY_1)$ to $Y$ lies in the kernel, denoted by $\HH^0(Y, \calL + mY_1)$, of the map
$$
\begin{tikzcd}[column sep=0pt]
\HH^0(\OO_{Y_1}(\calL_1 - mD)) \oplus \HH^0(\OO_{Y_2}(\calL_2 + mD-m\sum x_j)) \ar{d}
\\
\HH^0(\OO_D(\calL_2 + mD -m\sum x_j)) \ar[equal]{r} &
\HH^0(\OO_D(\calL_1 - mD))
\end{tikzcd}
$$
sending $(\gamma_1,\gamma_2)$ to $\gamma_1 - \gamma_2$.
We summarise the above discussion in the following theorem.

\begin{Theorem}
\label{thm: limiting rational curves}
Let $\pi: \calX\to B = \Spec \CC[[t]]$ be a flat proper family of surfaces whose generic fiber $\calX_\eta$ is a K3 surface and
whose central fiber $\calX_0 = Y = Y_1\cup Y_2$ is the union of two smooth rational surfaces $Y_i$ meeting transversely along a smooth anti-canonical curve $D$ in $Y_i$ for $i=1,2$. Suppose that $\calX$ is smooth outside of the $s$ distinct points $x_1, x_2, \ldots, x_s\in D$ satisfying \eqref{K3RATNOTESE000}. Let $\calL\in \Pic(\calX)$,
$f: C\to Y$ be a stable map of genus $g$ and $m$ be a non-negative integer satisfying
\begin{equation}\label{K3RATNOTESE026}
\begin{aligned}
h^0(\calX_\eta, \calL) &= h^0(Y, \calL + mY_1)\\
D \not\subset \Gamma &= f_* C \in \PP \HH^0(Y, \calL + mY_1)
\end{aligned}
\end{equation}
and \eqref{K3RATNOTESE025}-\eqref{K3RATNOTESE020}. Then after a finite base change, there exists a family of stable maps
$\phi: \mathscr{C}/B\to \calX/B$ such that
$\phi_0 = f$ and $\phi_* \mathscr{C}_\eta$ is an integral curve of geometric genus $g$ in $|\calL|$ on $\calX_\eta$.

If $f_G\circ \nu: \widehat{G}\to Y_i$ is an immersion
for the normalisation $\nu: \widehat{G}\to G$
of every component $G\subset C$, i.e.,
\begin{equation}\label{K3RATNOTESE046}
\begin{aligned}
&\begin{tikzcd}[column sep=large]
T_{\widehat{G}} \ar[hook]{r}{(f_G\circ \nu)_*} & f_G^* T_{Y_i}
\end{tikzcd}
\text{ is injective}
\\
&\hspace{12pt} \text{for all irreducible components } G\subset C,
\ f_G = f\Big|_G: G\to Y_i\\
&\hspace{24pt}\text{ and normalisation }
\nu: \widehat{G}\to G,
\end{aligned}
\end{equation}
then $\phi_\eta: \mathscr{C}_\eta\to \calX_\eta$ is an immersion.

If in addition to \eqref{K3RATNOTESE025}-\eqref{K3RATNOTESE046},
we assume
\begin{equation}\label{K3RATNOTESE031}
\begin{aligned}
& f_* C_i \text{ has normal crossings on } Y_i\backslash D
\text{ for }
C_i = f^{-1}(Y_i) \text{ and } i=1,2
\\
& f(p_1) \ne f(p_2) \text{ for all } p_1\ne p_2\in C_\times \text{ and}
\\
& f(G_1) \text{ and } f(G_2) \text{ meet transversely at } x_1,x_2,\ldots,x_s \text{ on } Y_i\\
&\hspace{24pt} \text{for all pairs of distinct components }
G_j \text{ with } f(G_j)\subset Y_i,
\end{aligned}
\end{equation}
then $\phi_* \mathscr{C}_\eta$ is nodal.
\end{Theorem}

\section{Nodal rational curves on generic K3 surfaces}
\label{sec: rigidifiers}

In this section, we use degenerations of type II of K3 surfaces as considered in the previous section to construct nodal
rational curves on {\em generic} surfaces inside moduli spaces of $\Lambda$-polarised K3 surfaces, generalising a
result of the first named author \cite{chen} to the higher rank case (cf.\ \cite{klcv} for similar type II degenerations
from higher rank lattices). 

\begin{Theorem}
 \label{thm: nodal curves}
Let $\Lambda$ be a lattice of rank two with intersection matrix
\begin{equation}\label{K3lattice-2m}
\begin{bmatrix}
2d & a\\
a & 2b
\end{bmatrix}_{2\times2}
\end{equation}
for some $a, b\in \ZZ$ and $d\in \ZZ^+$ satisfying $4bd-a^2<0$.
Let $M_\Lambda$ be the moduli space of $\Lambda$-polarised complex K3 surfaces and $L\in \Lambda$ such that $L$ is big and nef on a general K3 surface $X\in M_\Lambda$.
Then there exists an open and dense subset $U\subseteq M_\Lambda$ (with respect to the Zariski topology), depending on $L$, such that on every K3 surface $X\in U$,
the complete linear series $|L|$ contains an integral nodal rational curve if one of the following holds:
\begin{enumerate}
\item[A1.] $\det(\Lambda)$ is even;
\item[A2.] $L = L_1 + L_2 + L_3$ for some $L_i\in \Lambda$ satisfying that
$L L_i > 0$ and $L_i^2 > 0$ for $i=1,2,3$;
\item[A3.] $L = L_1 + L_2$ for some $L_i\in \Lambda$ satisfying that
$L L_i > 0$ for $i=1,2$, $L_1^2 > 0$, $L_2^2 = -2$, $L_1\not\in 2\Lambda$,
$L_1 - L_2\not\in n\Lambda$ for all $n\in\ZZ$ and $n\ge 2$,
and $L_1^2 + 2L_1L_2\ge 18$.
\end{enumerate}
\end{Theorem}

\begin{Remark}
 \label{rem: nodal curves}
Every even lattice of signature $(1,1)$ (the signature is dictated by the Hodge Index Theorem) is of the form
\eqref{K3lattice-2m}. For some special lattices of rank $2$, namely where $a$ is even and $b=0$, the above is due to Lewis and
the first named author \cite{C-L}.

Of course, the existence of rational curves on general K3 surfaces of Picard rank two implies the same on general K3
surfaces of Picard rank one. More precisely, the above theorem implies the existence of nodal rational curves in $|nL|$
on a general K3 surface with Picard lattice
$$
\begin{bmatrix}
2d
\end{bmatrix}_{1\times1}
$$
for all $n\in \ZZ^+$. This is due to the first named author when $d\ge 2$ \cite{chen}. The above theorem also resolves
the case $d=1$ in Picard rank one.

Case A3 is a technical extension required for the main theorem of \cite{regenerationinfinite} so can be ignored by the
casual reader.
\end{Remark}

\subsection{The proof of Theorem \ref{thm: nodal curves}}

We prove the theorem in three steps:

\begin{enumerate}
\item First, we embed a rank two K3 lattice \eqref{K3lattice-2m} into that of a K3 surface with many $(-2)$-curves: when $\det(\Lambda)$ is even, i.e., $a$ is even in \eqref{K3lattice-2m}, we embed
$\Lambda$ into a lattice with intersection matrix 
\begin{equation}
\label{K3RATNOTESE007}
\begin{bmatrix}
2\\
& -2\\
&& -2\\
&&& \ddots\\
&&&& -2
\end{bmatrix}_{(r+1)\times (r+1)}
\end{equation}
for some $r\leq6$; when $\det(\Lambda)$ is odd, i.e., $a$ is odd in \eqref{K3lattice-2m}, we embed
$\Lambda$ into a lattice with intersection matrix
\begin{equation}
\label{K3RATNOTESE019}
\begin{bmatrix}
0 & 1\\
1 & -2\\
&& -2\\
&&& \ddots\\
&&&& -2
\end{bmatrix}_{(r+1)\times (r+1)}
\end{equation}
for some $r\leq 4$. The embedding is itself a purely arithmetic problem. However, for our purposes, we also require the embedding to have the additional property that the image of a ``designated'' ample divisor $L$ remains (at least) big and nef, i.e., preserving a given polarisation. This introduces some extra complexity.

\item Second, we use the degeneration of K3 surfaces in Section \ref{sec: type II degeneration} to show the existence of nodal rational curves in almost all big and nef linear systems on a general K3 surface with Picard lattice \eqref{K3RATNOTESE007} or \eqref{K3RATNOTESE019}.

\item Third and finally, we deform a K3 surface $X_0$ with Picard lattice \eqref{K3RATNOTESE007} or \eqref{K3RATNOTESE019} to K3 surfaces $X_\eta$ with Picard lattice \eqref{K3lattice-2m} such that a nodal rational curve on $X_0$ deforms to a nodal rational curve on $X_\eta$.
\end{enumerate}

In summary, our argument involves two degenerations: the degeneration of K3 surfaces of Picard rank two to K3 surfaces with Picard lattices \eqref{K3RATNOTESE007} or \eqref{K3RATNOTESE019} and the degeneration of the latter to unions of rational surfaces.

In order to embed the lattice \eqref{K3lattice-2m} to \eqref{K3RATNOTESE007} or \eqref{K3RATNOTESE019},
we will make use of the classical result of Lagrange that every non-negative integer can be written as the square sum of
four integers and that of Legendre that every non-negative integer not of the form
of $4^a(8b + 7)$ can be written as the square sum of three integers.
However, in order to obtain a primitive embedding, as we will see, we require
these integers to be coprime. This is not possible in general but we can choose these integers such that their greatest common divisor is a power of $2$, a fact not in the standard formulation of these two theorems but implied by Dirichlet's proof of Legendre's theorem. So let us restate their theorems as follows:

\begin{Theoremx}[Lagrange--Legendre--Dirichlet]
Every positive integer $n$ not in the form of $4^a(8b + 7)$ for any $a,b\in \NN$\footnote{We use $\NN$ for the set of non-negative integers.} can be written as
\begin{equation}\label{K3RATNOTESE001}
n = m_1^2 + m_2^2 + m_3^2
\end{equation}
for some $m_1,m_2,m_3\in \NN$ with $\gcd(m_1,m_2,m_3) = 2^l$, where
$l\in \NN$ satisfies $4^l\mid n$ and $4^{l+1} \nmid n$.
As a consequence, every positive integer $n$
can be written as
\begin{equation}\label{K3RATNOTESE022}
n = m_1^2 + m_2^2 + m_3^2 + m_4^2
\end{equation}
for some $m_1,m_2,m_3, m_4\in \NN$ with
$\gcd(m_1,m_2,m_3,m_4) = 2^l$, where $l\in \NN$ satisfies $2^{2l+1} \mid n$ and $2^{2l+3}\nmid n$.
Furthermore, every positive integer can be written as the square sum of five coprime integers.
\end{Theoremx}

\begin{specialproof}[Dirichlet's Proof of Legendre's $3$-Square]
Let us outline Dirichlet's proof.
It is enough to prove \eqref{K3RATNOTESE001} for $n\equiv 1,2,3,5,6\ (\text{mod } 8)$.
The key is to find a {\em ternary quadratic form}
$$
F(x,y,z) = \begin{bmatrix}
x & y & z
\end{bmatrix} A \begin{bmatrix}
x\\
y\\
z
\end{bmatrix}
$$
such that $A$ is a $3\times 3$ positive definite symmetric integral matrix with $\det(A) = 1$ and
$F(x,y,z) = n$ has an integral solution $(x_0,y_0,z_0)$. If we can find such $A$, then there exists a matrix $P\in \text{SL}_3(\ZZ)$ such that $A = P^T P$ and hence
\begin{equation}
\label{K3RATNOTESE014}
\begin{bmatrix}
m_1\\
m_2\\
m_3
\end{bmatrix} = P \begin{bmatrix}
x_0\\
y_0\\
z_0
\end{bmatrix}
\end{equation}
is a solution of \eqref{K3RATNOTESE001} with $\gcd(m_1,m_2,m_3) = \gcd(x_0,y_0,z_0)$.
It turns out that we can choose
$$
A = \begin{bmatrix}
a_{11} & a_{12} & 1\\
a_{12} & a_{22} & 0\\
1 & 0 & n
\end{bmatrix}
$$
with integers $a_{ij}$ satisfying
\begin{equation}\label{K3RATNOTESE024}
a_{11} > 0,\ d = a_{11} a_{22} - a_{12}^2 > 0, \text{ and } a_{22} = dn-1.
\end{equation}
The corresponding $F(x,y,z) = n$ has an obvious solution $(x,y,z) = (0,0,1)$. So
$m_i$ given by \eqref{K3RATNOTESE014} are coprime as required.

To find $a_{ij}$ satisfying \eqref{K3RATNOTESE024}, we use the quadratic reciprocity law and
Dirichlet's Theorem on arithmetic progressions. We will skip this part of the proof.
\end{specialproof}

We start with the embedding of the lattice \eqref{K3lattice-2m} into \eqref{K3RATNOTESE007} when $\det(\Lambda)$ is even.
In the following, by a primitive lattice embedding we mean an injective lattice homomorphism with torsion-free cokernel. 

\begin{Lemma}\label{lem: lattice1}
For every even lattice $\Lambda$ of rank two, even determinant and signature $(1,1)$, there exists a positive integer
$r\leq 6$ so that $\Lambda$ can be primitively embedded into a lattice
$\Sigma_r$ with intersection matrix \eqref{K3RATNOTESE007}
$$
\begin{bmatrix}
2\\
& -2\\
&& -2\\
&&& \ddots\\
&&&& -2
\end{bmatrix}_{(r+1)\times (r+1).}
$$
\end{Lemma}

\begin{proof}
Such a lattice $\Lambda$ has intersection matrix
\begin{equation}
\label{K3RATNOTESE012}
\begin{bmatrix}
2a & 2m\\
2m & 2b
\end{bmatrix}
\end{equation}
for some integers $a,b,m$ satisfying $m^2 > ab$.
We claim that there exists a basis of $\Lambda$ such that
$a \ge 0$, $m\ge 0$ and $b < 0$ in \eqref{K3RATNOTESE012}.
It is easy to make $a\ge 0$ and $m\ge 0$. If $b < 0$, we are done. Otherwise,
let us consider all bases of $\Lambda$ whose intersection matrices \eqref{K3RATNOTESE012} satisfies $a \ge 0$, $m \ge 0$ and $b\ge 0$. Let us choose a basis $\{F_1, F_2\}$
among these bases that minimises the trace $2(a+b)$ of the matrix \eqref{K3RATNOTESE012}. Without loss of generality, let us assume that $b \ge a \ge 0$.
Since $m^2 > ab$, $m > a = F_1^2$. Then the intersection matrix of
$\{F_1, F_2 - F_1\}$ is
$$
\begin{bmatrix}
2a & 2(m - a)\\
2(m - a) & 2a+2b - 4m
\end{bmatrix}
$$
with $b > a+b - 2m$. By our choice of $\{F_1,F_2\}$, we must have $2a+2b - 4m < 0$, which proves our claim.

Let us assume that $\Lambda$ is generated by $F_1$ and $F_2$ such that
$$
F_1^2 = 2a \ge 0,\  F_1 F_2 = 2m\ge 0 \text{ and } F_2^2 = 2b < 0.
$$

When $2\nmid a$, we let
\begin{align*}
\sigma(F_1) &= \frac{a+1}2 (A - E_6) + E_6\\
\sigma(F_2) &= m (A - E_6) - \sum_{i=1}^5 m_i E_i\\
\text{ with }
-b &= \sum_{i=1}^5 m_i^2\text{ and } \gcd(m_1,m_2,m_3,m_4,m_5) = 1
\end{align*}
be the embedding $\sigma: \Lambda\hookrightarrow \Sigma_6$,
where $A, E_1, E_2, \ldots, E_r$ are the generators of $\Sigma_r$ with intersection matrix
\eqref{K3RATNOTESE007}.

When $2\mid a$, we let
\begin{align*}
\sigma(F_1) &= \frac{a+2}2 (A -E_6) - E_1 + E_6\\
\sigma(F_2) &= (m+m_1) (A - E_6) - \sum_{i=1}^5 m_i E_i\\
\text{with }-b &= \sum_{i=1}^5 m_i^2\text{ and } \gcd(m_1,m_2,m_3,m_4,m_5) = 1.\qedhere
\end{align*}
\end{proof}

\begin{Remark}\label{rem: lattice1.1}
A K3 surface $X$ with Picard lattice \eqref{K3RATNOTESE007} can be realised as a double cover $\varphi: X \to S$, where
$S$ is a del~Pezzo surface of degree $9-r$ and $\varphi$ is ramified along a general curve in $|-2K_S|$.

We have the pullback map
$\varphi^*: \Pic(S) \xrightarrow{\sim} \Pic(X)$ on the Picard groups such that $(\varphi^* L)^2 = 2L^2$ for all $L\in \Pic(S)$.
Therefore, $\varphi^*$ induces an isomorphism of nef cones of $S$ and $X$. Recall that the effective cone of curves on $S$ is generated by $(-1)$-curves for $2\le r\le 8$. Correspondingly,
the effective cone of curves on $X$ is generated by $(-2)$-curves.
It is also useful to us that there are lattice automorphisms
$\sigma_{i_1i_2i_3}$ of $\Pic(X)$ given by
$$
\begin{aligned}
\sigma_{i_1i_2i_3}(A) &= 2A - E_{i_1} - E_{i_2} - E_{i_3}\\
\sigma_{i_1i_2i_3}(E_{i_1}) &= A - E_{i_2} - E_{i_3}\\
\sigma_{i_1i_2i_3}(E_{i_2}) &= A - E_{i_3} - E_{i_1}\\
\sigma_{i_1i_2i_3}(E_{i_3}) &= A - E_{i_1} - E_{i_2}\\
\sigma_{i_1i_2i_3}(E_i) &= E_i \text{ when } i\ne i_1,i_2,i_3
\end{aligned}
$$
for $1\le i_1<i_2<i_3\le r$. These are induced by
the Cremona (or quadratic) transformations of $\Pic(S)$.
Together with the symmetric group acting on $\{E_i\}$, $\sigma_{i_1i_2i_3}$ generate the subgroup
$$
\Aut(\Pic(X))^+\subset \Aut(\Pic(X))
$$
that preserves the nef cone of $X$.
The action of $\Aut(\Pic(X))^+$ on the set of $(-2)$-curves on $X$ is transitive for $2\le r \le 8$.
\end{Remark}

\begin{Remark}\label{rem: lattice1.2}
The bound $r\le 6$ in Lemma \ref{lem: lattice1} is optimal. For example, in the case $8\mid a$, $b=0$ and $4\mid m$ in \eqref{K3RATNOTESE012}, a primitive embedding $\sigma: \Lambda\hookrightarrow \Sigma_r$ must be in the form of
$$
\begin{aligned}
\sigma(F_1) &= m_r(A - E_r) + m E_r - \sum_{i=1}^{r-1} m_i E_i\\
\sigma(F_2) &= A - E_r\\
\text{ with }
2mm_r - m^2 - a &= \sum_{i=1}^{r-1} m_i^2\text{ and } \gcd(m_1,m_2,\ldots,m_{r-1}) = 1
\end{aligned}
$$
after composing $\sigma$ with an action of $\Aut(\Sigma_r)$.
Since $8\mid (m^2 + 2mm_r - a)$, $\gcd(m_1,m_2,\ldots,m_{r-1}) \ne 1$ if $r \le 5$. So we need $r=6$.

On the other hand, there are situations that we can embed $\Lambda$ to $\Sigma_5$. For example, when $8\nmid b$, we can always write $-b$ as the square sum of four coprime integers. So the construction in the proof works for $r=5$ when $8\nmid b$.
\end{Remark}

Next, let us embed the lattice \eqref{K3lattice-2m} into \eqref{K3RATNOTESE019} when $\det(\Lambda)$ is odd.

\begin{Lemma}\label{lem: lattice2}
Every even lattice $\Lambda$ of rank two, odd determinant and signature $(1,1)$ can be primitively embedded into a lattice $\Sigma_r$ with intersection matrix \eqref{K3RATNOTESE019}
$$
\begin{bmatrix}
0 & 1\\
1 & -2\\
&& -2\\
&&& \ddots\\
&&&& -2
\end{bmatrix}_{(r+1)\times (r+1)}
$$
for some $r\le 4$.
\end{Lemma}

\begin{proof}
As in the proof of Lemma \ref{lem: lattice1}, we can find a basis $\{ F_1, F_2\}$ of
$\Lambda$ such that
$$
F_1^2 = 2a \ge 0,\  F_1 F_2 = m > 0 \text{ and } F_2^2 = 2b < 0
$$
for some integers $a,b,m$ with $2\nmid m$.
We can always find $m_1\in \ZZ^+$ such that $(m-bm_1)m_1 > a$ and
$(m-bm_1)m_1 - a$ is not in the form of $4^\alpha(8\beta + 7)$ for any $\alpha,\beta\in \NN$. Then we let
$$
\begin{aligned}
\sigma(F_1) &= (m-bm_1)A + m_1(A+E_1)  - \sum_{i=2}^4 m_i E_i\\
\sigma(F_2) &= bA + (A+E_1)\\
\text{ with }
(m-bm_1)m_1 - a &= \sum_{i=2}^4 m_i^2 \text{ with } \gcd(m_2,m_3,m_4) = 2^l
\end{aligned}
$$
where $A, E_1, E_2,\ldots,E_r$ the generators of $\Sigma_r$ with intersection matrix \eqref{K3RATNOTESE019}.
\end{proof}

\begin{Remark}\label{rem: lattice2}
Let $X$ be a K3 surface with Picard lattice \eqref{K3RATNOTESE019}.
We claim that the effective cone of $X$ is generated by the $(-2)$-curves
\begin{equation}\label{K3RATNOTESE028}
E_i \text{ and } P_j = A - E_j \hspace{24pt} \text{ for }
1\le i \le r\text{ and } 2\le j\le r
\end{equation}
and $L = d A + m_1 E_1 - \sum_{i=2}^r m_i E_i$ is nef if and only if
$LE_i\ge 0$ and $LP_j\ge 0$, or equivalently,
\begin{equation}\label{K3RATNOTESE013}
d \ge 2m_1 \ge 4m_j\ge 0 \hspace{24pt} \text{ for } 2\le j\le r
\end{equation}
when $2\le r \le 5$.

Clearly, all curves in \eqref{K3RATNOTESE028} are $(-2)$-curves and
the inequalities in \eqref{K3RATNOTESE013} are necessary for $L = d A + m_1 E_1 - \sum m_i E_i$ to be nef. On the other hand, \eqref{K3RATNOTESE013} guarantees that $L^2 \ge 0$. Therefore, $X$ does not contain $(-2)$-curves other than those in \eqref{K3RATNOTESE028} and the inequalities in \eqref{K3RATNOTESE013} are also sufficient for $L$ to be nef.
\end{Remark}

Neither of the above lemmas produces an embedding $\sigma$ preserving a given polarisation $L$.
It turns out that we can always compose an existing $\sigma: \Lambda\hookrightarrow \Sigma$ with a lattice automorphism $\alpha\in \Aut(\Sigma)$ such that $\alpha\circ \sigma(L)$ is big and nef.

\begin{Lemma}\label{lem: lattice3}
Suppose that there exists a primitive embedding $\Lambda\hookrightarrow \Sigma$ between K3 Picard lattices
$\Lambda$ and $\Sigma$.
Then for each $L\in \Lambda$ with $L^2 > 0$, there is a primitive embedding
$\sigma: \Lambda\hookrightarrow \Sigma$ such that $\sigma(L)$ is big and nef on $X$, when $\Sigma$ is identified with the Picard lattice of a projective K3 surface $X$. Moreover, fixing another class $C\in \Lambda$, we can choose $\sigma$ such that $\sigma(NL - C)$ is big and nef on $X$ for $N$ sufficiently large.
\end{Lemma}

\begin{proof}
Fixing an ample divisor $D$ on $X$, we consider all primitive embeddings $\sigma: \Lambda\hookrightarrow \Sigma$ satisfying that $\sigma(L) . D > 0$. We choose $\sigma$ among these embeddings such that $\sigma(L) . D$ achieves the minimum.
So $\sigma(L)$ is pseudo-effective. By the Zariski Decomposition, we can write
$$
\sigma(L) = P + N
$$
where $P$ is a nef $\QQ$-divisor, $N$ is a $\QQ$-effective divisor whose components have negative self-intersection matrix and $PN = 0$. If $\sigma(L)$ is nef, i.e., $\sigma(L) = P$, we are done. Otherwise, there exists an integral curve $R\subset X$ such that $\sigma(L) . R < 0$. Then
$R\subset \mathop{\mathrm{supp}}\nolimits(N)$ is a $(-2)$-curve on $X$.

We let $\alpha: \Sigma\to \Sigma$ be the group
homomorphism given by
\begin{equation}\label{K3RATNOTESE027}
\alpha(F) = F + (F.R) R
\end{equation}
for $F\in \Sigma$. Note that $\alpha^2 = \text{id}$ and $(\alpha(F))^2  = F^2$ for all $F\in \Sigma$. So $\alpha$ is a lattice automorphism.

Let $\widehat{\sigma} = \alpha\circ \sigma$.
Since $\widehat{\sigma}(L) . P = P^2 > 0$
and $L^2 > 0$, $\widehat{\sigma}(L)$ is big and hence
$\widehat{\sigma}(L) . D > 0$. On the other hand, since $\sigma(L) . R < 0$,
$$
\widehat{\sigma}(L) . D = \sigma(L) . D + (\sigma(L).R) RD < \sigma(L) . D
$$
which contradicts our hypothesis that $\sigma$ minimises $\sigma(L) . D$. So
$\sigma(L) = P$ is big and nef.

Let $E = \sigma(C)$ and let $R_1, R_2, \ldots, R_l$ be the integral curves on $X$ such that $PR_i=0$ for $i=1,2,\ldots,l$.
Let us consider the set
$$
\begin{aligned}
& \Pi = \Big\{
\xi: \Lambda\hookrightarrow \Sigma \text{ primitive embedding } \big|\ \xi(L) = P \text{ and}
\\
&\hspace{72pt} \xi(C) = E + m_1 R_1 + m_2 R_2 + \ldots + m_l R_l \text{ for some } m_i\in \ZZ
\Big\}.
\end{aligned}
$$
Since $R_1, R_2, \ldots, R_l$ have negative definite self-intersection, there are only finitely many
$(m_1, m_2,\ldots, m_l)\in \ZZ^l$ satisfying that
$$
(E+m_1 R_1 + m_2 R_2 + \ldots + m_l R_l)^2 = E^2.
$$
Therefore, there exists $\xi\in \Pi$ maximising $\xi(C) . D$. We claim that $\xi(C) . R_i \le 0$. Otherwise, suppose that $\xi(C) . R >0$ for some $R=R_i$. Then $\widehat{\xi} = \alpha\circ \xi\in \Pi$ and
$$
\widehat{\xi}(C) . D = \xi(C) . D + (\xi(C).R) RD > \xi(C) . D
$$
which contradicts our choice of $\xi$. Therefore, $\xi(C) . R >0$ for all integral curves $R$ with $PR=0$.
Replacing $\sigma$ by $\xi$, we see that $\sigma(NL - C)$ is big and nef for $N$ sufficiently large.
\end{proof}

Now we have embedded the lattice \eqref{K3lattice-2m}
into \eqref{K3RATNOTESE007} or \eqref{K3RATNOTESE019}.
Next, we want to prove the existence of nodal rational curves on K3 surfaces with Picard lattices \eqref{K3RATNOTESE007} and \eqref{K3RATNOTESE019}. Here we use the Type II degeneration of K3 surfaces introduced in the previous section. It turns out that in order to produce rational curves on a general K3 surface, we need to construct rational curves on a log K3 surface with some tangency conditions. More precisely, we want to find rational curves on a del~Pezzo surface satisfying some tangency conditions with a fixed anti-canonical curve.

\begin{Definition}
For a Cartier divisor $A$ on a projective surface $X$, we use the notation $V_{A,g}$ to denote the Severi variety of integral curves of geometric genus $g$ in $|A|$.
For a curve $D\subset X$ and a zero cycle $\alpha = m_1 p_1 +
m_2 p_2 + \ldots + m_l p_l \in Z_0(D)$, we use the notation $V_{A,g,D,\alpha}$ to denote the subvariety of
$V_{A,g}$ consisting of integral curves $C \in |A|$ of genus $g$ with the properties that
$C$ meets $D$ properly and
there exist $q_i \in \nu^{-1} (p_i)$ and $n_i \ge m_i$ such that $q_1, q_2,\ldots,q_l$ are distinct and
$\nu^* D =
n_i q_i$ when $\nu$ is restricted to the open neighborhoods of $p_i$ and $q_i$ for $i = 1, 2, \ldots, l$,
where $\nu : \widehat{C}\to X$ is the normalisation of $C$, $m_1, m_2, \ldots, m_l \in \NN$ and $p_1, p_2, \ldots, p_l$ are
points on $D$ such that $D$ is locally Cartier at each $p_i$.
\end{Definition}

The variety $V_{A,g,D,m_1p_1 + m_2 p_2 + \ldots + m_l p_l}$ parametrises the curves of fixed tangencies with $D$. We can also define the subvariety of $V_{A,g}$ of curves of moving tangencies with $D$ by letting some of $p_i$ moving. For example,
with $p_1,p_2,\ldots,p_s$ moving, these curves are parametrised by
$$
\bigcup_{(p_1,p_2,\ldots,p_s)\in (D^s)^*} V_{A,g,D,m_1p_1 + m_2 p_2 + \ldots + m_l p_l}
$$
where $(D^s)^*$ is the open set of $D^s = D^{\times s}$ of points $p_i\ne p_j$ for $1\le i < j \le l$.  
In the following we write $A\ge B$ or $B\le A$ if $A-B$ is effective.

\begin{Theorem}
\label{thm: log K3}
Let $X$ be a smooth projective complex rational surface containing a smooth anti-canonical curve $D\in |-K_X|$. Let $A_1,A_2,\ldots,A_n$ be divisors on $X$ such that
\begin{itemize}
\item $A_i D \ge 1$, $A_1 D \ge 2$,
\item $(A_1 + A_2 + ... + A_j) A_{j+1} \ge 1$ and
\item $V_{A_i, 0} \ne \emptyset$
\end{itemize}
for $i=1,2,\ldots,n$ and $j=1,\ldots,n-1$. Then for $A = A_1 + A_2 + \ldots + A_n$, all distinct points $p_1,p_2,\ldots,p_l$ on $D$ satisfying
\begin{equation}
\label{K3RATNOTESE021}
\begin{aligned}
&\text{there does not exist an integral curve } B\subset X \text{ such that } A\ge B \text{ and}\\
&\hspace{72pt} B\cap D \subset \{p_1,p_2,\ldots,p_l\}
\end{aligned}
\end{equation}
and all $m_1, m_2, \ldots,m_l\in\NN$ satisfying $m = \sum m_i \le AD - 1$,
there exists an effective divisor $G$ on $X$ such that $DG=0$ and
$$
V_{A-G,0,D,m_1p_1 + m_2p_2 + \ldots + m_l p_l} \ne \emptyset.
$$

Moreover, for $V=V_{A-G,0,D,m_1p_1 + m_2p_2 + \ldots + m_l p_l}$ and a general member $C$ in $V$, we have
\begin{enumerate}
\item $\dim V = AD - m - 1$.
\item If $m \le AD - 2$,
the normalisation $\nu: \widehat{C}\to X$ of $C$ induces an injection
$\nu_*: T_{\widehat{C}} \to \nu^* T_X$, and $(CD)_{p_i} = m_i$
for $i=1,2,\ldots,l$.
\item If $m = AD - 1$, (2) holds for $p_1\in D$ general.
\item If $m \le AD -3$,
$C$ meets a fixed reduced curve $F\subset X$ transversely outside of $\{p_1,p_2,\ldots,p_l\}$.
\item If $m = AD - 3 + s$ for some $s\in \NN$,
(4) holds for $p_1,p_2,\ldots,p_s\in D$ general and $m_1,m_2,\ldots,m_s\in \ZZ^+$.
\item If $m\le AD - 4$, all singularities of $C$ are of type $\CC[[x,y]]/(x^a - y^a)$, i.e., ordinary.
\item If $m = AD - 4 + s$ for some $s\in \NN$,
(6) holds for $p_1,p_2,\ldots,p_s\in D$ general and $m_1,m_2,\ldots,m_s\in \ZZ^+$.
\item If $m \le AD - 5$, $C$
is nodal.
\item If $m = AD - 5 + s$ for some $s\in \NN$, (8) holds for $p_1,p_2,\ldots,p_s\in D$ general and $m_1,m_2,\ldots,m_s\in \ZZ^+$.
\end{enumerate}
\end{Theorem}
\begin{proof}
By the standard deformation theory of curves on surfaces \cite{harrismorrison}, $V_{A_i,0}$ has the expect dimension $A_i D - 1$. Since $\dim V_{A_1,0} = A_1 D - 1\ge 1$, a general member $C_1\in V_{A_1,0}$ meets a fixed reduced curve $F\subset X$ transversely.
In particular, for a fixed $C_2\in V_{A_2,0}$, $C_2$ meets the normalization of $C_1$ transversely. Let us consider $C = C_1\cup C_2$ and the stable map $\nu: C^\nu\to X$ that normalises all singularities of $C$ except one point among $C_1\cap C_2$. The deformation space of $\nu$ has dimension at least
$(A_1+A_2) D - 1$. On the other hand, $\dim V_{A_1,0} + \dim V_{A_2,0} = (A_1+A_2) D - 2$. So $C$ deforms to an integral rational curve in $|A_1+A_2|$. Consequently, $V_{A_1+A_2,0}$ is nonempty of expected dimension $\dim V_{A_1+A_2,0} = (A_1+A_2) D - 1$. We may continue to apply the same argument to $C_{12}\cup C_3$, where $C_{12}$ is a general member of $V_{A_1+A_2,0}$ and $C_3\in V_{A_3,0}$. Eventually, we conclude that $V_{A,0}$ is nonempty of the expected dimension $AD - 1$.

Next let us prove $V_{A-G,0,D,m_1p_1 + m_2p_2 + \ldots + m_l p_l} \ne \emptyset$ by induction on $m$.
There is nothing to do when $m = 0$.
Suppose that
$$
V = V_{A-G,0,D,m_1p_1 + m_2p_2 + \ldots + m_l p_l} \ne \emptyset
$$
for some $m\le AD-2$. It suffices to show that
\begin{equation}
\label{K3RATNOTESE038}
V_{A-G-G',0,D,(m_1+1)p_1 + m_2p_2 + \ldots + m_l p_l} \ne\emptyset
\end{equation}
for some $G'\ge 0$ and $DG' = 0$.

Note that $V$ has the expected dimension $AD - m - 1$. Let $\overline{V}$ be the closure of $V$ in
$|A-G|$.
Let $g: \Gamma\hookrightarrow \overline{V}$ be an integral projective curve passing through a general point of $V_{A-G,0,D,mp}$. After a finite base change, there exists a family $f: \mathscr{C}\to X$ of stable maps of genus $0$ over $\Gamma$ such that $f_* \mathscr{C}_b = g(b)$ for every point $b\in \Gamma$. We may also choose $f$ such that $f^{-1}(D)$ is a union of sections over $\Gamma$ and some ``vertical'' components. That is,
\begin{equation}\label{K3RATNOTESE011}
f^* D = m_1 P_1 + m_2 P_2 + \ldots + m_l P_l +
n_1 Q_1 + n_2 Q_2 + \ldots + n_a Q_a + W
\end{equation}
where $P_i$ and $Q_j$ are sections of $\pi: \mathscr{C}\to \Gamma$,
$f(P_i) = p_i$ and $\pi_* W = 0$. Since all components of $\mathscr{C}_b$ are rational and $D$ is a smooth elliptic curve, $f_* W = 0$.

On the other hand, every connected component of $f^{-1}(D)$ must dominate $D$. Since $f_* P_1 = 0$, $P_1$ must lie on the same connected component as some $Q_j$. Therefore, there exists
$W' \subset W$ such that $P_1\cup W' \cup Q_j$ is connected. So
$P_1$ and $Q_j$ are joined by a chain of components contained in $f^{-1}(p_1)\cap\mathscr{C}_b$ for some point $b\in \Gamma$. In an open neighborhood $U\subset \mathscr{C}_b$ of the connected component of $f^{-1}(p_1)\cap \mathscr{C}_b$ containing
$P_1\cap \mathscr{C}_b$, we have
$$
(f_* U . D)_{p_1} \ge m_1+1.
$$
Let us write
\begin{equation}
\label{K3RATNOTESE015}
f_* \mathscr{C}_b = \mu_1 C_1 + \mu_2 C_2 + \ldots + \mu_r C_r + G'
\end{equation}
where $G'$ is supported on the components of $f_* \mathscr{C}_b$ that are disjoint from $D$ and $C_j$ are the components satisfying $C_j\in V_{C_j,0,D,\alpha_j}$ for effective $0$-cycles $\alpha_j$ on $D$
satisfying
$$
\begin{aligned}
&\mathop{\text{supp}}\nolimits \alpha_j \subset \{
p_1,p_2,\ldots,p_l\}
\text{ and}\\
& \sum_{j=1}^r \mu_j \alpha_j \ge (m_1+1)p_1 + m_2p_2 + \ldots + m_l p_l.
\end{aligned}
$$
Due to our choice of $p_1,p_2,\ldots,p_l$ in \eqref{K3RATNOTESE021},
$p_1,p_2,\ldots,p_l$ cannot be the only intersections between $C_j$ and $D$. Therefore, $\deg \alpha_j \le C_jD - 1$.

Since $\Gamma$ is an arbitrary curve in $\overline{V}$ passing through a general point of $V$, the above argument shows that there is a codimension one subvariety $Z$ of
$\overline{V}$ parametrising the curves \eqref{K3RATNOTESE015}:
\begin{equation}
\label{K3RATNOTESE039}
\begin{aligned}
Z &= \Big\{
\mu_1 C_1 + \mu_2 C_2 + \ldots + \mu_r C_r + G'\in \overline{V}: DG'= 0,\\
&\hspace{48pt} C_j\in V_{C_j,0,D,\alpha_j},\ \deg \alpha_j \le C_jD - 1 \text{ for } j=1,2,\ldots,r,\\
&\hspace{48pt} \sum_{j=1}^r \mu_j \alpha_j \ge (m_1+1)p_1 + m_2p_2 + \ldots + m_l p_l
\Big\}\\
\dim Z &= \dim V - 1 = AD - m - 2.
\end{aligned}
\end{equation}
Since $\deg \alpha_j \le C_jD - 1$, we obtain $\dim V_{C_j,0,D,\alpha_j}\le C_jD - \deg \alpha_j - 1$. There are at most
countably many rational curves that are disjoint from $D$, so $G'$ is rigid. Therefore,
$$
\begin{aligned}
AD - m - 2 &= \dim Z \le \sum_{j=1}^r \dim V_{C_j,0,D,\alpha_j}
\\
&\le \sum_{j=1}^r \mu_j \dim V_{C_j,0,D,\alpha_j} = \sum_{j=1}^r \mu_j(C_jD - \deg\alpha_j - 1)\\
&= A D - \sum_{j=1}^r \mu_j \deg\alpha_j - \sum_{j=1}^r \mu_j \le AD - m - 1 - \sum_{j=1}^r \mu_j.
\end{aligned}
$$
Then we must have $\sum \mu_j = 1$. That is,
$\sum \mu_j C_j$ contains only one component $C=C_1$ with multiplicity one. Clearly,
$$
C\in V_{A-G-G',0,D,(m_1+1)p_1 + m_2p_2 + \ldots + m_l p_l}
$$
and \eqref{K3RATNOTESE038} follows.
Once we have $V_{A-G,0,D,m_1p_1 + m_2p_2 + \ldots + m_l p_l} \ne \emptyset$, all the other statements follow from the standard deformation theory of curves on surfaces \cite{harrismorrison}.
\end{proof}

\begin{Corollary}\label{cor: log K3}
Let $X$ be a complex del~Pezzo surface and $D\in |-K_X|$ be a smooth anti-canonical curve on $X$. Then for all big and
nef divisors $A$ on $X$, all points $p_1,p_2,\ldots,p_l\in D$ satisfying \eqref{K3RATNOTESE021} and all $m_1, m_2,
\ldots,m_l\in\NN$ satisfying $m = \sum m_i \le AD - 1$, $$V_{A,0,D,m_1p_1 + m_2p_2 + \ldots + m_l p_l}\ne \emptyset.$$
\end{Corollary}
\begin{proof}
By Theorem \ref{thm: log K3}, it suffices to show that $V_{A,0}\ne \emptyset$. Note that $D$ is ample and there does not exist $G\ge 0$, $G\ne 0$ and $DG = 0$.

It should be a well-known fact that every big and nef complete linear series on a del~Pezzo surface contains an integral
rational curve, but we include a simple argument proving this. It suffices to write $A = A_1 + A_2 + \ldots + A_n$ such that
$A_iD \ge 1$, $A_1D \ge 2$ and $V_{A_i,0}\ne \emptyset$ as in Theorem \ref{thm: log K3}.

We may assume that $A$ is ample. Otherwise, we simply blow down the $(-1)$-curves disjoint from $A$ to obtain $f: X\to Y$. Then $f_*A$ is ample on $Y$ and $A = f^*(f_* A)$.

Let $\Lambda, E_1, E_2,\ldots,E_r$ be the effective divisors generating $\Pic(X)$ with $\Lambda^2 = 1$, $\Lambda E_i=0$ and $E_i^2 = -1$ for $i=1,2,\ldots,r$.

When $r=0$, $A = d \Lambda$ for some $d\in \ZZ^+$ with $\Lambda D = 3$ and $V_{\Lambda,0}\ne \emptyset$, obviously. So $V_{A,0}\ne \emptyset$.

When $r=1$, $A = d\Lambda + m (\Lambda - E_1)$ for some $d,m\in \ZZ^+$ with
$V_{\Lambda,0}\ne \emptyset$ and $V_{\Lambda-E_1,0}\ne \emptyset$. So $V_{A,0}\ne \emptyset$.

When $r\ge 2$, the effective cone of curves of $X$ is generated by $(-1)$-curves on $X$. Therefore, there exists $m\in \ZZ^+$ such that $A - mD = G$ is nef while $A - (m+1)D$ is not.
It is not hard to see that $V_{D,0}\ne \emptyset$.

If $G^2 > 0$, then by our choice of $m$, $GR = 0$ for some $(-1)$-curve $R$.
We can blow down the $(-1)$-curves disjoint from $G$ to obtain $f:X\to Y$ such that
$f_* G$ is ample and $G = f^* (f_* G)$. So by induction on $\rank \Pic(X)$, we conclude that
$V_{G,0} \ne \emptyset$. Obviously, $G D \ge 2$ by the Hodge Index Theorem. Therefore, $V_{A,0}\ne \emptyset$.

If $G\ne 0$ and $G^2 = 0$, then $G = a F$ for some indivisible $F\in \Pic(X)$ with $F$ base point free and $F^2 = 0$. It is not hard to see that
$V_{F,0} \ne \emptyset$ and $FD = 2$. Therefore, $V_{A,0}\ne \emptyset$.

If $G=0$ and $D^2 \ge 2$, we can again derive $V_{A,0}\ne \emptyset$ from $V_{D,0}\ne \emptyset$. So the only case left is that
$A = mD$ and $D^2 = 1$, i.e., $r=8$ and $A = -mK_X$ for some $m\ge 2$. In this case, we can write
$$
2 D = \underbrace{(3\Lambda - 2E_1 - \sum_{i=2}^7 E_i)}_{R_1} +
\underbrace{(3\Lambda - \sum_{i=2}^7 E_i - 2E_8)}_{R_2}
$$
where $R_1$ and $R_2$ are $(-1)$-curves with $R_1R_2 = 3$. By deforming the union $R_1\cup R_2$,
we conclude that $V_{2D,0}\ne\emptyset$. Thus, $V_{A,0}\ne \emptyset$.
\end{proof}

The appearance of the effective divisor $G$ in
Theorem \ref{thm: log K3}, which we will call a {\em $(-2)$-tail}, is quite inconvenient for us. In the case that $X$ is a del~Pezzo surface, we automatically have $G = 0$ due to the ampleness of $-K_X$. However, we also need to apply the theorem to the case that $-K_X$ is big and nef. Namely, we need to work with singular del~Pezzo surfaces. In this case, every connected component of $G$, if nonzero, is supported on a tree of smooth $(-2)$-curves. Fortunately, the following proposition guarantees that $(-2)$-tails do not appear in the flat limits of integral rational curves.

\begin{Proposition}
\label{prop: (-2)-tails}
Let $\calX$ be a smooth proper family of surfaces over $B=\Spec \CC[[t]]$
and $f: \mathscr{C}/B\to \calX/B$ be a family of stable maps over $B$ such that
\begin{itemize}
\item the geometric generic fiber $\overline{\mathscr{C}}_\eta$ of $\mathscr{C}/B$ is connected and smooth and $f$ maps $\mathscr{C}$ birationally onto its image;
\item the image of the central fiber of $\mathscr{C}_0$
of $\mathscr{C}/B$ under $f$ is
$$
f_* \mathscr{C}_0 = C_0 + m_1 C_1 + m_2 C_2 + \ldots + m_r C_r
$$
where $m_i\in \NN$, $C_1,C_2,\ldots,C_r$ are smooth rational curves satisfying that
$C_i^2 \le -2$ and $C_1 + C_2 + \ldots + C_r$ has simple normal crossings and $C_0$ is a (possibly reducible and non-reduced) curve meeting $C_1 + C_2 + \ldots + C_r$ transversely on $X=\calX_0$;
\item each curve $C_i$ deforms in the family $\calX/B$
for $i=0,1,\ldots,r$;
\item $\mathscr{C}_0 - f^{-1}(C_1\cup C_2\cup \ldots\cup C_r)$ and $\mathscr{C}_0$ have the same arithmetic genus.
\end{itemize}
Then $m_1 = m_2 = \ldots = m_r = 0$.
\end{Proposition}
\begin{proof}
Let $M = \mathscr{C}_0 - f^{-1}(C_1\cup C_2\cup \ldots\cup C_r)$. The fact that $M$ and $\mathscr{C}_0$ have the same arithmetic genus
is equivalent to saying that every connected component $T$ of
$f^{-1}(C_1\cup C_2\cup \ldots\cup C_r)$ is a tree of smooth rational curves and $TM = 1$.

Suppose that at least one of $m_i$ is positive. We will construct a (possibly infinite) sequence $\Gamma_0, \Gamma_1, \ldots, \Gamma_n, \ldots$ such that
\begin{itemize}
\item $\Gamma_0 = M$ and each $\Gamma_i$ is either $M$ or an irreducible component of $\mathscr{C}_0$ dominating one of $C_1,C_2,\ldots,C_r$;
\item for each $i\in \NN$, $\Gamma_i \ne \Gamma_{i+1}$ and there exist a point $p\in C_1\cup C_2\cup \ldots\cup C_r$
and a connected component $T_i$ of $f^{-1}(p)$ satisfying that $T_i\cap \Gamma_i \ne \emptyset$
and $T_i \cap \Gamma_{i+1} \ne \emptyset$;
\item $T_i\ne T_{i+1}$ for all $i\in \NN$;
\item the sequence terminates at $n\ge 2$ if and only if $\Gamma_n = M$.
\end{itemize}
Once we have such a sequence, we must have $\Gamma_i = \Gamma_j$ for some $j-i\ge 2$. Then it is easy to see that the dual graph of $\mathscr{C}_0$ contains a path $G_1G_2\ldots G_m$ such that $G_1,G_2,\ldots,G_m$ are distinct components of $\mathscr{C}_0$ for some $m\ge 2$, $G_i\subset \mathscr{C}_0 - M$ for $1<i<m$,
$G_j\cap G_{j+1}\ne \emptyset$ for $j=1,2,\ldots,m-1$,
and $G_1G_2\ldots G_m$ is either a circuit or $G_1$ and $G_m$ are two distinct components of $M$. Either way,
this contradicts the hypothesis that $M$ and $\mathscr{C}_0$ have the same arithmetic genus. So it suffices to produce the sequence $\{\Gamma_i \}$ with the above properties.

Let $C = C_a$ for some $1\le a \le r$. Let $\calX^{[1]}$ be the blowup of $\calX$ along $C$. The central fiber of $\calX_0^{[1]}$ of $\calX^{[1]}$ over $B$ is the union of the proper transform of $X$, which we still denote by $X$, and the exceptional divisor $R_1$ meeting transversely along $X\cap R_1 = D_1$. One of the key hypotheses is that $C$ deforms in the family $\calX/B$. So the normal bundle of $C$ in $\calX$ splits as
$$
\calN_{C/\calX} = \calN_{X/\calX}\Big|_C \oplus \calN_{C/X} = \OO_C \oplus \OO_C(-n_a)
$$
where $C^2 = C_a^2 = -n_a\le -2$. Consequently, $R_1 \cong \FF_{n_a}$ for $n_a\ge 2$. And since $D_1^2 = n_a$ on $R_1$, $R_1$ contains a section $D_2$ over $C$ with $D_2^2 = -n_a$ and $D_1\cap D_2 = \emptyset$.

We continue to blow up $\calX^{[1]}$ along $D_2$ to obtain $\calX^{[2]}$. Then the central fiber $\calX_0^{[2]}$ of $\calX^{[2]}/B$ is the union $X\cup R_1 \cup R_2$, where $X$ is the proper transform of $X\subset \calX$, $R_1$ is the proper transform of $R_1\subset \calX^{[1]}$, $R_2$ is the exceptional divisor,
$X$ and $R_1$ meet transversely along $D_1 = X\cap R_1$, $R_1$ and $R_2$ meet transversely along $D_2 = R_1\cap R_2$ and
$X\cap R_2 = \emptyset$. Here we again abuse the notations by using $X, R_i, D_j$ for the subvarieties of
all $\calX^{[k]}$. Again, we have $R_2 \cong \FF_{n_a}$ and a section $D_3$ of $R_2/C$ with
$D_3^2 = -n_a$ on $R_2$ and $D_2\cap D_3=\emptyset$. We may continue to blow up
$\calX^{[2]}$ along $D_3$ to obtain $\calX^{[3]}$. So we have a sequence of blowups
\begin{equation}\label{K3RATNOTESE037}
\begin{tikzcd}
\calX = \calX^{[0]} & \calX^{[1]} \ar{l} & \calX^{[2]} \ar{l} & \ldots \ar{l} & \calX^{[l]}\ar{l}
\end{tikzcd}
\end{equation}
where $\calX_0^{[l]} = X\cup R_1\cup R_2 \cup \ldots \cup R_l$ such that
\begin{itemize}
    \item $X$ is the proper transform of $X = \calX_0$,
    \item $R_i\cong \FF_{n_a}$ for $i=1,2,\ldots,l$,
    \item $R_i\cap R_j = \emptyset$ for $0\le i < j-1 \le l-1$ and $R_0=X$,
    \item $R_{i-1}$ and $R_i$ meet transversely along $D_i = R_{i-1}\cap R_i$ and
    $D_i^2=-n_a$ on $R_{i-1}$ and $D_i^2 = n_a$ on $R_i$ for $i=1,2,\ldots,l$.
\end{itemize}

Over a general point $q\in C$, the map $f: \mathscr{C}_0\to X$ is finite and unramified onto its image if
$C\subset f(\mathscr{C})$. Therefore, the proper transform of $f(\mathscr{C})$ under $\calX^{[l]}\to \calX$ does not contain $D_i$ for $i=1,2,\ldots,l$. And since $f(\mathscr{C})$ is irreducible and $C_0 \ne C$,
for $l$ large enough, the proper transform of $f(\mathscr{C})$ does not contain $D_{l+1}$, either, where
$D_{l+1}$ is the section of $R_l/C$ with $D_{l+1}^2 = -n_a$.
Let us choose $l$ with this property and also lift
$f: \mathscr{C}\to \calX$ to a family $\widehat{f}: \mathscr{C} \to \calX^{[l]}$ of stable maps with the diagram
$$
\begin{tikzcd}
\widehat{\mathscr{C}} \ar{r}{\widehat{f}} \ar{d}[left]{\varphi} & \calX^{[l]}\ar{d}\\
\mathscr{C} \ar{r}{f} & \calX
\end{tikzcd}
$$
after a base change.

For every component $\Gamma$ of $\mathscr{C}_0$ that dominates $C$ via $f$, our choice of $l$ implies that $\widehat{f}(\widehat{\Gamma})$ lies in $R_i$ for some $1\le i\le l$ and $D_i, D_{i+1}\not\subset \widehat{f}(\widehat{\Gamma})$, where
$\widehat{\Gamma}\subset \widehat{\mathscr{C}}_0$ is the proper transform of $\Gamma$ under $\varphi$. Let us define two things using $\widehat{f}$:
\begin{enumerate}
\item We define a partial order among the components of $\mathscr{C}_0$ that dominates $C$ via $f$. Let $\Gamma$ and $\Gamma'$ be two components of $\mathscr{C}_0$ dominating $C$. Let $\widehat{\Gamma}$ and $\widehat{\Gamma}'\subset \widehat{\mathscr{C}}_0$ be their proper transforms under $\varphi$.
Suppose that $\widehat{f}(\widehat{\Gamma})\subset R_i$ and
$\widehat{f}(\widehat{\Gamma}')\subset R_j$
for some $1\le i,j\le l$.
We say that $\Gamma\prec \Gamma'$ or $\Gamma' \succ \Gamma$ if
$i < j$ and $\Gamma\nprec \Gamma'$ or $\Gamma' \nsucc \Gamma$
if $i\ge j$.
\item Let $\Gamma$ be a component of $\mathscr{C}_0$ that dominates $C$ via $f$ and $\widehat{\Gamma} \subset \widehat{\mathscr{C}}_0$ be its proper transform. Suppose that $\widehat{f}(\widehat{\Gamma})\subset R_i$ for some $1\le i \le l$. We define $\xi_\Gamma$ to be the effective $0$-cycle on $\Gamma$ given by
$$
\xi_\Gamma = \varphi_* ( (\widehat{f}^* R_{i-1}) . \widehat{\Gamma}).
$$
Note that $\widehat{f}(\widehat{\Gamma})$ is an integral curve on $R_i\cong \FF_{n_a}$ meeting $D_{i-1}$ and $D_i$ properly. Therefore, we have
\begin{equation}\label{K3RATNOTESE041}
\deg \xi_\Gamma = (\widehat{f}_* \widehat{\Gamma}) . R_{i-1} \ge n_a \deg_\Gamma (f) \ge 2\deg_\Gamma (f)
\end{equation}
where $\deg_\Gamma(f)$ is the degree of the map $f: \Gamma\to C$.
\end{enumerate}

One of our basic tools is the following observation:
\begin{quote}
($*$) Let $V\subset \calX^{[l]}$ be an \'etale/analytic/formal open neighborhood of a point $p\in D_i = R_{i-1}\cap R_i$ for some $1\le i \le l$ such that
$$
V\cong \CC[[x,y,z,t]]/(xy-t^m).
$$
Let $U\subset \widehat{\mathscr{C}}$ be a connected component of $\widehat{f}^{-1}(V)$. We write
$$
U_0 = W_{i-1} + W_i
$$
with $\widehat{f}(W_{i-1}) \subset R_{i-1}$ and $\widehat{f}(W_i) \subset R_i$. Then
$$
\widehat{f}_* W_{i-1} . R_i = \widehat{f}_* W_i . R_{i-1}.
$$
\end{quote}

We will construct the sequence $\{\Gamma_i\}$ inductively such that for each $i\in \ZZ^+$,
either $\Gamma_i = M$ or $\Gamma_i$ dominates some $C_a$ via $f$ and
\begin{equation}\label{K3RATNOTESE040}
\mathop{\text{supp}}(\xi_{\Gamma_i}) \not\subset T_{i-1}.
\end{equation}

We have $\Gamma_0 = M$. Let us first find $\Gamma_1$. Since $\mathscr{C}_0$ is connected, there exist a point
$p\in C_0\cap C_a$ for some $1\le a \le r$, a connected component $T_0$ of $f^{-1}(p)$
and a component $\Gamma_1$ of $\mathscr{C}_0$ dominating $C_a$
such that $T_0\cap \Gamma_0\ne\emptyset$ and $T_0\cap \Gamma_1 \ne\emptyset$.
Since $C_0$ meets $C_a$ transversely at $p$, we must have
$$
v_q(\xi_{\Gamma_1}) = 1
$$
by ($*$), where $q = T_0\cap \Gamma_1$ and $v_q(\xi_{\Gamma_1})$ is the multiplicity of $q$ in the $0$-cycle
$\xi_{\Gamma_1}$. By \eqref{K3RATNOTESE041}, $\mathop{\text{supp}}(\xi_{\Gamma_1})$ contains at least another point $q'\ne q$. So \eqref{K3RATNOTESE040} holds for $i=1$. We have found $\Gamma_1$ with the required property.

Suppose that we have found $\Gamma_i$. If $\Gamma_i = M$, the sequence terminates and we are done. Suppose that $\Gamma_i$ dominates $C_a$ for some $1\le a\le r$. By \eqref{K3RATNOTESE040}, there is a point $q\in \mathop{\text{supp}}(\xi_{\Gamma_i})$ such that $q\not\in T_{i-1}$. Let $T_i$ be the connected component of $f^{-1}(f(q))$ such that $q = T_i\cap \Gamma_i$. There are three cases:
\begin{enumerate}
\item $M\cap T_i \ne \emptyset$. In this case, we simply let $\Gamma_{i+1} = M$.
\item There is a component $\Gamma$ of $\mathscr{C}_0$ dominating $C_a$ such that
$\Gamma\cap T_i\ne \emptyset$ and $\Gamma\prec \Gamma_i$. Then we let $\Gamma_{i+1} = \Gamma$ since we have
$$
\mathop{\text{supp}}(\xi_{\Gamma}) \not\subset T_i
$$
by ($*$).
\item Both (1) and (2) fail. By ($*$), there must be a component $G$ of $\mathscr{C}_0$ dominating
$C_b$ for some $1\le b\ne a \le r$ such that $G\cap T_i\ne\emptyset$. This case requires more effort.
\end{enumerate}

Now let us deal with case (3). Since both (1) and (2) fail, $M\cap T_i = \emptyset$ and for all components
$\Gamma\ne\Gamma_i$ of $\mathscr{C}_0$ dominating $C_a$ and satisfying $\Gamma\cap T_i \ne \emptyset$, we have
$\Gamma\nprec \Gamma_i$.

Let $P$ be the union of the components $\Gamma$ of $\mathscr{C}_0$ dominating $C_a$ and satisfying $\Gamma\cap T_i \ne \emptyset$ and let $Q$ be the union of the components $G$ of $\mathscr{C}_0$ dominating $C_b$ and satisfying $G\cap T_i \ne \emptyset$. We let $U$ be an \'etale open neighborhood of $T_i$ in $\mathscr{C}$ and let $f_U$ be the restriction of $f$ to $U$. Then by ($*$), we have
\begin{equation}\label{K3RATNOTESE044}
\deg_P(f_U)\le \deg_Q (f_U)
\end{equation}
where $\deg_P (f_U)$ and $\deg_Q(f_U)$ are the degrees of the maps
$$
f: P\cap U\to C_a\cap f(U) \text{ and } f: Q\cap U \to C_b \cap f(U),
$$
respectively.
We claim that there exists at least one component $G\subset Q$ such that
\begin{equation}\label{K3RATNOTESE045}
\mathop{\text{supp}}(\xi_G) \not\subset T_i.
\end{equation}
Otherwise, suppose that $\mathop{\text{supp}}(\xi_G) \subset T_i$ for all components $G\subset Q$.
And since $G$ and $T_i$ meet at a unique point $s$, this implies
that $\mathop{\text{supp}}(\xi_G)$ consists of the single point $s$,
the map $f:G\to C_b$ is totally ramified at $s$ and
\begin{equation}\label{K3RATNOTESE043}
v_s(\xi_G) = \deg \xi_G \ge 2 \deg_G(f)
\end{equation}
by \eqref{K3RATNOTESE041}.

Then by \eqref{K3RATNOTESE043} and by applying ($*$) to the blowup sequence \eqref{K3RATNOTESE037} over $C = C_b$, we conclude that
\begin{equation}\label{K3RATNOTESE042}
\deg_P(f_U) = \sum_{\substack{G\subset Q\\
    s=G\cap T_i}} v_s(\xi_G) \ge 2\deg_Q(f) = 2\deg_Q (f_U)
\end{equation}
where $\deg_Q(f) = \deg_Q(f_U)$ since the map $f:G\to C_b$ is totally ramified at $G\cap T_i$ for all components $G\subset Q$.
Clearly, \eqref{K3RATNOTESE044} and \eqref{K3RATNOTESE042} contradict each other. This proves \eqref{K3RATNOTESE045} for some component $G\subset Q$. So it suffices to take $\Gamma_{i+1} = G$.
\end{proof}

\begin{Corollary}
\label{cor: log k3 (-2)-tail}
Under the same hypotheses of Theorem \ref{thm: log K3}, we further assume that $DP > 0$ for all nef and effective divisors $P\not\supset D$. Then Theorem \ref{thm: log K3} holds for $G = 0$.
\end{Corollary}
\begin{proof}
Let $\Sigma$ be the union of all rational curves $R\subset X$ such that $DR = 0$. We claim that $\Sigma$ is a union of smooth rational curves with negative definite self-intersection matrix.

Suppose that $R_1, R_2,\ldots,R_n\subset \Sigma$ are rational curves
whose self-intersection matrix is not negative definite. We may choose $\{R_1,R_2,\ldots,R_n\}$ such that every proper subset of
$\{R_1,R_2,\ldots,R_n\}$ has negative definite self-intersection matrix. Since the self-intersection matrix of $\{R_1,R_2,\ldots,R_n\}$ is not negative definite, we can find $c_1,c_2,\ldots,c_n\in \ZZ$, not all zero, such that
$(c_1 R_1 + c_2 R_2 + \ldots + c_n R_n)^2 \ge 0$. We may choose $c_i$ such that at least one of $c_i$ is positive. Let us write
$$
c_1 R_1 + c_2 R_2 + \ldots + c_n R_n = \underbrace{\sum_{c_i > 0} c_i R_i}_A
- \underbrace{\sum_{c_i \le 0} (-c_i)R_i}_B.
$$
We claim that $B=0$; otherwise, $A^2 < 0$, $B^2 < 0$ and $AB \ge 0$ by our hypothesis on $R_i$ and hence $(A-B)^2 < 0$.
Therefore, $B=0$ and $c_1,c_2,\ldots,c_n > 0$.
In other words, there exists an effective divisor
$A = \sum c_i R_i$ supported on $R_1+R_2+\ldots+R_n$ such that
$A^2 \ge 0$. Let us choose $A$ such that $B^2 < 0$ for all
$0< B < A$. Clearly, $A$ is nef; otherwise,
$AR_i \le -1$ for some $i$ and then
$$
(A-R_i)^2 = A^2 - 2AR_i + R_i^2 \ge A^2 + 2 - 2 = A^2 \ge 0.
$$
So $A$ is nef and hence $DA > 0$, which is a contradiction.

In conclusion, all subsets $\{R_1,R_2,\ldots,R_n\}\subset \Sigma$
have negative definite self-intersection matrices. This actually implies that $\Sigma$ is a union of finitely many smooth rational $(-2)$-curves with simple normal crossings.

In the proof of Theorem \ref{thm: log K3}, the support of $G'$ in \eqref{K3RATNOTESE015} is contained in $\Sigma$. Hence $G' = 0$ by Proposition \ref{prop: (-2)-tails}. Thus, Theorem \ref{thm: log K3} holds for $G=0$.
\end{proof}

\begin{Theorem}
\label{K3NOTESTHM001}
For a general complex K3 surface $X$ with Picard lattice
\eqref{K3RATNOTESE007}
$$
\begin{bmatrix}
2\\
& -2\\
&& -2\\
&&& \ddots\\
&&&& -2
\end{bmatrix}_{(r+1)\times (r+1)}
$$
$r\leq 8$ and a big and nef divisor $L$ on $X$,
there exists an integral rational curve $C\in |L|$ such that the normalisation $\nu: \widehat{C}\to X$ of $C$ induces an injection $\nu_*: T_{\widehat{C}} \to \nu^* T_X$.
In addition, if $r\le 6$, $C$ can be chosen to be nodal.
\end{Theorem}
\begin{proof}
We consider a type II degeneration $Y = Y_1\cup Y_2$, where $Y_i$ are two del Pezzo surfaces
whose Picard groups are generated by effective divisors $A_i, E_{i1}, E_{i2},\\ \ldots, E_{ir}$
with intersection matrix
\begin{equation}\label{K3RATNOTESE008}
\begin{bmatrix}
1\\
& -1\\
&& -1\\
&&& \ddots\\
&&&& -1
\end{bmatrix}_{(r+1)\times (r+1)}
\end{equation}
for $i=1,2$ and $Y_1$ and $Y_2$ meet transversely along a smooth anti-canonical curve $D = Y_1\cap Y_2$. Let $\imath_i:D\hookrightarrow Y_i$ be the inclusion.
We further require
\begin{equation}\label{K3RATNOTESE009}
\imath_1^* (A_1) \,=\,
\imath_2^* (A_2)
\text{ \quad and \quad }
\imath_1^* E_{1j} = \imath_2^* E_{2j}
\end{equation}
in $\Pic(D)$ for $j=1,2,\ldots,r$.

For a general choice of such $Y$,
the relations from \eqref{K3RATNOTESE009} are the only relations among
$\imath_i^* A_i$ and $\imath_i^* E_{ij}$ in $\Pic(D)$.
If these are satisfied, then $\Pic(Y)$ is freely generated by $A$ and $E_j$
whose restriction to $Y_i$ are $A_i$ and $E_{ij}$, respectively.
By Theorem \eqref{K3RATNOTESTHM000},
$Y$ can be deformed to a K3 surface with Picard lattice \eqref{K3RATNOTESE007}.
Clearly, $E_j$ deform to disjoint $(-2)$-curves and $A$ deforms to a big and nef divisor orthogonal
to $E_j$ correspondingly.

Let $\pi: {\calX}\to \Spec \CC[[t]]$ be such a family with ${\calX}_0 = Y$.
Now we use $A, E_1, E_2, \ldots, E_r$ to denote the effective divisors on $\calX$
whose restrictions to $Y_i$ are $A_i, E_{i1}, E_{i2}, \ldots, E_{ir}$, respectively, for $i=1,2$.
Meanwhile, the big and nef divisor $L$ on the generic fiber $\calX_\eta$ extends to a divisor, which we still denote by $L$, on $\calX$. We let $L_i$ be the restriction of $L$ to $Y_i$ for $i=1,2$.

Clearly, the $3$-fold $\calX$ has $18 - 2r$ rational double points
$x_1$, $x_2$, \ldots, $x_{18-2r}$ on $D$ satisfying
\begin{equation}
\label{K3RATNOTESE010}
\begin{aligned}
\OO_D(x_1+x_2+\ldots+x_{18-2r}) &= \calN_{D/Y_1}\otimes \calN_{D/Y_2}
\\
&=\OO_D(-K_{Y_1})\otimes \OO_D(- K_{Y_2})\\
&= \OO_D(6A - 2E_1 - 2E_2 - \ldots - 2E_r),
\end{aligned}
\end{equation}
which is the only relation among $x_1$, $x_2$,\ldots, $x_{18-2r}$
for a general choice of ${\calX}$.

To find a rational curve in $|L|$ on the generic fiber $\calX_\eta$ of $\calX$, it suffices to locate a ``limiting rational curve''
$\Gamma$ in $|L|$ on $\calX_0$.

Suppose that $LD \ge 2$. By Corollary \ref{cor: log K3}, $V_{L_i,0,D,mp}\ne \emptyset$ for $p\in D$ general and $m = LD - 1$. And since $x_1$ is a general point on $D$, we can find
rational curves $\Gamma_i\in |L_i|$ for $i=1,2$ such that
\begin{equation}
\label{K3RATNOTESE017}
\Gamma_i . D = x_1 + m p
\end{equation}
on $Y_i$, $\Gamma_i$ is smooth at $p$ and the normalisation $\nu: \widehat{\Gamma}_i \to Y_i$ of $\Gamma_i$ induces an injection $\nu_*: T_{\widehat{\Gamma}_i} \to \nu^* T_{Y_i}$, i.e.,
$\nu$ is an immersion.

By Theorem \ref{thm: limiting rational curves}, $\Gamma = \Gamma_1\cup \Gamma_2$ can be deformed to a rational curve $\mathscr{C}_\eta$ on the generic fiber ${\calX}_\eta$ of $\calX$ after a finite base change. The normalisation
$\nu: \widehat{\mathscr{C}}_\eta\to {\calX}_\eta$ of $\mathscr{C}_\eta$ is an immersion since the same holds for $\Gamma_i$, the point $x_1\in \Gamma_1\cap \Gamma_2$ deforms to a node and
the point $p\in \Gamma_1\cap \Gamma_2$ deforms to $m-1$ nodes of $\mathscr{C}_\eta$.

If $LD = 1$, this only happens when $r = 8$ and $L_i = -K_{Y_i} = D$. For a general del~Pezzo surface $Y_i$ of degree $1$, there exists a nodal rational curve $\Gamma_i$ in $|-K_{Y_i}|$ that meets $D$ transversely at a unique point $p$. Then it is easy to see that $\Gamma = \Gamma_1\cup \Gamma_2$ can be deformed to a nodal rational curve $\mathscr{C}_\eta$ on the generic fiber ${\calX}_\eta$.

Suppose that $r\le 6$. By the Hodge Index Theorem, $LD \ge 3$.

If $LD \le 4$, it is easy to see that $(K_{Y_i} + L_i) L_i\le 0$. Namely, the arithmetic genus of $L_i$ is at most $1$.
Then a general member of $V_{L_i,0}$ must be nodal since
its normalisation is an immersion.
So there is a nodal rational curve in $|L_i|$ passing through
the $a = LD - 1$ general points $x_1,x_2,\ldots,x_a$ on $D$.
Thus, we may find $\Gamma = \Gamma_1\cup \Gamma_2$ such that $\Gamma_i$ are nodal rational curves in $|L_i|$ satisfying
\begin{equation}\label{K3RATNOTESE018}
\Gamma_i . D = x_1 + x_2 + \ldots + x_a + p
\end{equation}
on $Y_i$ for $i=1,2$. Then $\Gamma = \Gamma_1\cup \Gamma_2$ can be deformed to a nodal rational curve $\mathscr{C}_\eta$ on the generic fiber ${\calX}_\eta$.

If $L D \ge 5$, then by Theorem \ref{thm: log K3} and Corollary \ref{cor: log K3}, $V_{L_i,0,D,mp} \ne \emptyset$ and a general member of $V_{L_i,0,D,mp}$ is nodal for $m = LD - 4$ and a general point $p\in D$. And since
$x_1, x_2, x_3, x_4$ are four general points on $D$ by \eqref{K3RATNOTESE010}, we can find nodal rational curves
$\Gamma_i \in V_{L_i,0,D,mp}$ such that
\begin{equation}\label{K3RATNOTESE029}
\Gamma_i . D = x_1 + x_2 + x_3 + x_4 + m p
\end{equation}
on $Y_i$ for $i=1,2$. Then $\Gamma = \Gamma_1\cup \Gamma_2$ can be deformed to a nodal rational curve $\mathscr{C}_\eta$ on the generic fiber ${\calX}_\eta$.
\end{proof}

\begin{Theorem}
\label{K3NOTESTHM002}
Let $X$ be a general complex K3 surface with Picard lattice
\eqref{K3RATNOTESE019}
$$
\begin{bmatrix}
0 & 1\\
1 & -2\\
&& -2\\
&&& \ddots\\
&&&& -2
\end{bmatrix}_{(r+1)\times (r+1)}
$$
generated by effective divisors $A, E_1, E_2, \ldots, E_r$ for $r\leq 5$
and let $L$ be a big and nef divisor on $X$ satisfying
\begin{equation}\label{K3RATNOTESE050}
\left\{
\begin{aligned}
LA &\ge 3\\
L E_5 &\le 2 \hspace{12pt} \text{if } r = 5.
\end{aligned}
\right.
\end{equation}
Then there exists an integral nodal rational curve $\Gamma\in |L|$. Moreover, there exist integral nodal rational curves $P$ and $Q$ in $|A|$ and
$|4A + 2E_1 - E_2 - \ldots - E_r|$, respectively, such that
$\Gamma + P + Q$ has normal crossings on $X$.
\end{Theorem}
\begin{proof}
By the description of the nef cone of $X$ (see Remark \ref{rem: lattice2}) and \eqref{K3RATNOTESE050}, we have
\begin{equation}\label{K3RATNOTESE030}
\begin{aligned}
L &= d A + m_1 E_1 - m_2 E_2 - \ldots - m_r E_r
\\
&\text{for }
d,m,m_i\in \NN,\
d\ge 2m_1 \ge 4 \max_{2\le i\le r} m_i,\text{ and } m_1\ge 3\\
&\hspace{24pt} m_5 \le 1 \hspace{12pt} \text{if } r=5.
\end{aligned}
\end{equation}

We let $Y_1$ be a smooth projective rational surface with Picard lattice
$$
\begin{bmatrix}
0 & 1\\
1 & -2\\
&& -1\\
&&& -1\\
&&&& \ddots\\
&&&&& -1
\end{bmatrix}_{2r\times 2r}
$$
generated by effective divisors $A_1, B_1, G_1,G_2,\ldots,G_{2r-2}$ and let $D$ be a smooth anti-canonical curve on $Y_1$. We further require
\begin{equation}\label{K3RATNOTESE033}
\OO_D(A_1) = \OO_D(G_1 + G_2) = \OO_D(G_3 + G_4) = \ldots = \OO_D(G_{2r-3} + G_{2r-2}).
\end{equation}
Such $Y_1$ can be realised as the blowup of $\FF_2$ at $2r-2$ points $p_1,p_2,\ldots,p_{2r-2}$ such that
$p_{2i-1}$ and $p_{2i}$ lie on the same fiber of $\FF_2$ over $\PP^1$ for $i=1,2,\ldots,r-1$.

We let $Y_2\cong \PP^1\times \PP^1$ with $\Pic(Y_2)$
generated by two rulings $A_2, B_2$ and let $D$ be a smooth
anti-canonical curve on $Y_2$.

Let $Y=Y_1\cup Y_2$ be the union of $Y_1$ and $Y_2$ glued transversely along $D$ satisfying
\begin{equation}\label{K3RATNOTESE032}
\OO_D(\imath_1^* A_1) =
\OO_D(\imath_2^* A_2)
\end{equation}
where $\imath_i:D\hookrightarrow Y_i$ is the inclusion for $i=1,2$.

Note that such $Y=Y_1\cup Y_2$ is not projective. But we can deform it to a projective K3 surface
whose Picard lattice has rank $r+2$ and contains the lattice \eqref{K3RATNOTESE019} as a primitive sublattice. That is,
there exists a flat and proper (but non-projective) family $\pi: {\calX}\to \Spec \CC[[t]]$
of surfaces such that ${\calX}_0 = Y$ and the generic fiber $\calX_\eta$ of $\calX$ is
a K3 surface whose Picard lattice has rank $r+2$ and contains \eqref{K3RATNOTESE019} as a primitive sublattice. This follows from Theorem \ref{thm: k3 type2 degeneration}.

There are effective divisors $A, E_1, E_2, \ldots, E_r$ on $\calX$ such that
$$
\begin{aligned}
\OO_{Y_1}(A) &= \OO_{Y_1}(A_1),\ \OO_{Y_2}(A) = \OO_{Y_2}(A_2),\\
\OO_{Y_1}(E_1) &= \OO_{Y_1}(B_1),\ \OO_{Y_2}(E_1)= \OO_{Y_2}\\
\OO_{Y_1}(E_i) &= \OO_{Y_1}(G_{2i-3} + G_{2i-2}),\ \OO_{Y_2}(E_i) = \OO_{Y_2}(A_2)\text{ for }
i=2,\ldots,r
\end{aligned}
$$
The $3$-fold $\calX$ has $18 - 2r$ rational double points
$x_1$, $x_2$, \ldots, $x_{18-2r}$ on $D$ satisfying
\begin{equation}
\label{K3RATNOTESE035}
\begin{aligned}
\OO_D(x_1+x_2+\ldots+x_{18-2r}) &= \calN_{D/Y_1}\otimes \calN_{D/Y_2}
\\
&=
\OO_D(-K_{Y_1})\otimes \OO_D(- K_{Y_2})\\
&= \OO_D((7-r) A) \otimes \OO_D(2B_2),
\end{aligned}
\end{equation}
which is the only relation among $x_1$, $x_2$,\ldots, $x_{18-2r}$
for a general choice of ${\calX}$.

Let $L$ be the divisor on $\calX$ defined by \eqref{K3RATNOTESE030}. As before, to prove the existence of rational curves in $|L|$ on the generic fiber $\calX_\eta$, it suffices to find a limiting rational curve in $|L|$ on $\calX_0 = Y$. However, due to the fact that $L$ is not big when restricted to $Y_2$, we cannot construct such a curve in $|L|$ on $Y$. To overcome this, we need to work with the ``twisted'' linear series $|L + Y_1|$ on $\calX$.

As explained in Section \ref{sec: type II degeneration},
$\HH^0(Y, L + Y_1)$ is the kernel of the map
$$
\begin{tikzcd}[column sep=0pt]
\HH^0(\OO_{Y_1}(L_1 - D)) \oplus \HH^0(\OO_{Y_2}(L_2 + D-\sum x_j)) \ar{d}\\
\HH^0(\OO_D(L_2 + D -\sum x_j)) \ar[equal]{r}
& \HH^0(\OO_D(L_1 - D))
\end{tikzcd}
$$
sending $(\gamma_1,\gamma_2)$ to $\gamma_1 - \gamma_2$,
where $L_i$ are the restrictions of $L$ to $Y_i$ for $i=1,2$ and given explicitly by
$$
\begin{aligned}
L_1 &= L\Big|_{Y_1} = d A_1 + m_1 B_1 - \sum_{j=2}^r m_j (G_{2j-3} + G_{2j-2})\\
L_2 &= L\Big|_{Y_2} = (d-\sum_{j=2}^r m_j) A_2.
\end{aligned}
$$

By a direct computation, $h^0(\calX_\eta, L) = h^0(Y, L+Y_1)$. So every $(\gamma_1,\gamma_2)$ in $\HH^0(Y, L + Y_1)$ can be deformed to a section in $\HH^0(L)$ on the generic fiber $\calX_\eta$. It suffices to find a limiting rational curve $\Gamma\subset Y$ cut out by such $\gamma_i$.

Without loss of generality, let us assume that $m_2\ge \ldots\ge m_r$. Suppose that $m_i = 0$ for $i > a$ and $m_i > 0$ for $i \le a$.
We have
$$
\begin{aligned}
L_1 - D &= (d-4) A_1 + (m_1-2) B_1 - \sum_{j=2}^r (m_j-1) (G_{2j-3} + G_{2j-2})\\
&= M + \sum_{j=a+1}^r (G_{2j-3} + G_{2j-2})\\
L_2 + D &= (d+2-\sum_{j=2}^r m_j) A_2 + 2B_2
\\
\text{for } M &= (d-4) A_1 + (m_1-2) B_1 - \sum_{j=2}^a (m_j-1) (G_{2j-3} + G_{2j-2}).
\end{aligned}
$$
Since $d-4 \ge 2(m_1 -2)\ge 4(m_2-1)\ge \ldots\ge 4(m_a - 1)$ by \eqref{K3RATNOTESE030}, we can write
\begin{equation}
\label{K3RATNOTESE036}
\begin{aligned}
M &= (m_a - 1)(2A_1 + B_1 - \sum_{j=2}^a G_{2j-3})\\
&\quad +(m_a - 1)(2A_1 + B_1 - \sum_{j=2}^a G_{2j-2})\\
&\quad + \sum_{i=2}^{a-1} (m_i-m_{i+1}) (2A_1 + B_1 - \sum_{j=2}^i G_{2j-3})\\
&\quad + \sum_{i=2}^{a-1} (m_i-m_{i+1}) (2A_1 + B_1 - \sum_{j=2}^i G_{2j-2})\\
&\quad + (m_1-2m_2)(2 A_1 + B_1) + (d-2m_1)A_1
\end{aligned}
\end{equation}
and conclude that $V_{M,0} \ne \emptyset$. Similarly, $V_{L_2+D-cA_2,0}\ne \emptyset$ for all $c\le 4$.
We let
$$
\lambda = \min(4, MD -1) \text{ and }
m = MD - \lambda.
$$

If $MD \le 4$, it is easy to see by \eqref{K3RATNOTESE036} that
the arithmetic genus of $M$ is at most $1$. So a general member of $V_{M,0}$ is nodal and there exists a nodal rational curve in $|M|$ passing through $\lambda$ general points on $D$.
If $MD\ge 5$, by Corollary \ref{cor: log k3 (-2)-tail},
$$
V_{M,0,D,mp}\ne \emptyset
\text{ and }
V_{L_2+D-4A_2,0,D,mp}\ne \emptyset
$$
and general members of $V_{M,0,D,mp}$ and $V_{L_2+D-4A_2,0,D,mp}$ are nodal for $p\in D$ general. So we may find
$\Gamma\subset Y$ such that
$$
\begin{aligned}
\Gamma &= \Gamma_1 \cup \Gamma_2 \cup \ldots\cup \Gamma_\lambda\cup \Gamma_{\lambda+1}\cup \Gamma_{\lambda+2}\cup \Gamma_{\lambda+3},\\
\text{where } &\Gamma_1,\Gamma_2,\ldots,\Gamma_{\lambda}, \Gamma_{\lambda+2}\subset Y_2,
\ \Gamma_{\lambda+1}, \Gamma_{\lambda+3}\subset Y_1,\\
\Gamma_{\lambda+1} &\in V_{M,0,D,mp},\ \Gamma_1,\Gamma_2,\ldots,\Gamma_{\lambda}\in |A_2|\\
\Gamma_{\lambda+2} &\in V_{L_2 + D - \lambda A_2,0,D,mp}\\
\Gamma_{\lambda+3} &= G_{2a-1} \cup G_{2a-2}\cup \ldots \cup G_{2r-3} \cup G_{2r-2}
\\
\Gamma_1.D &= x_1 + y_1
\\
\Gamma_2.D &= x_2 + y_2
\\
\vdotswithin{\Gamma_2.D} &= \vdotswithin{x_2 + y_2}
\\
\Gamma_\lambda.D &= x_\lambda + y_\lambda\\
\Gamma_{\lambda+1}.D &= y_1 + y_2 + \ldots + y_\lambda + mp\\
\Gamma_{\lambda+2}.D &= mp + w_1 + w_2 + \ldots + w_{2r-2a} + x_{\lambda+1} + x_{\lambda+2}
+ \ldots + x_{18-2r}\\
\Gamma_{\lambda+3}.D &= w_1 + w_2 + \ldots + w_{2r-2a}.
\end{aligned}
$$
Here we choose $\Gamma_{\lambda+1}$ and $\Gamma_{\lambda+2}$ to be the general members of $V_{M,0,D,mp}$
and $V_{L_2+D-\lambda A_2,0,D,mp}$, respectively. So they are nodal, as explained above, in both cases $MD \le 4$ and $MD \ge 5$.
Therefore,
\begin{itemize}
\item $\Gamma_{\lambda+1} + \Gamma_{\lambda+3}$ and
$\Gamma_1 + \Gamma_2 + \ldots + \Gamma_{\lambda} + \Gamma_{\lambda+2}$
have normal crossings on $Y_i$,
\item
$\Gamma_{\lambda+1} + \Gamma_{\lambda+3}$ and
$\Gamma_1 + \Gamma_2 + \ldots + \Gamma_{\lambda} + \Gamma_{\lambda+2}$
meet $D$ transversely outside of $p$ on $Y_i$, and
\item
$\Gamma_{\lambda+1} + \Gamma_{\lambda+3}$ and
$\Gamma_1 + \Gamma_2 + \ldots + \Gamma_{\lambda} + \Gamma_{\lambda+2}$
have simple tangencies with $D$ at $p$ on $Y_i$ for $i=1,2$.
\end{itemize}

By Theorem \ref{thm: limiting rational curves},
we see that $\Gamma$ can be deformed to a nodal rational curve in $|L|$ on $\calX_\eta$.
To construct a nodal rational curve in $|A|$, we let
$$
\begin{aligned}
P &= P_1 \cup P_2,\ \text{where } P_1\subset Y_1, \ P_2\subset Y_2,\\
P_1 &\in |A_1|,\ P_2\in |A_2|\\
P_1.D &= x_{\lambda+1} + q
\\
P_2.D &= x_{\lambda+1} + q.
\end{aligned}
$$
Again by Theorem \ref{thm: limiting rational curves},
we see that $P$ can be deformed to a nodal rational curve in $|A|$ on $\calX_\eta$.
To construct a nodal rational curve in $|4A+2E_1 - E_2 - \ldots -E_r|$, we let
$$
\begin{aligned}
Q &= Q_1 \cup Q_2\cup \ldots\cup Q_{6-r},\ \text{where } Q_1\subset Y_1, \ Q_2, \ldots, Q_{6-r}\subset Y_2,\\
Q_1 &\in |4A_1+2B_1 - G_1 - \ldots - G_{2r-2}|,\ Q_2\in |A_2|\\
Q_1.D &= s_2 + s_3 + \ldots + s_{6-r}
\\
Q_2.D &= x_{\lambda+2} + s_2\\
Q_3.D &= x_{\lambda+3} + s_3\\
\vdotswithin{Q_2.D} &= \vdotswithin{x_{\lambda+2} + s_2}\\
Q_{6-r}.D &= x_{\lambda+6-r} + s_{6-r}
\end{aligned}
$$
where $Q_1$ is a nodal rational curve in $|4A_1+2B_1 - G_1 - \ldots - G_{2r-2}|$ passing through the general points $s_2,\ldots,s_{6-r}$.
Again by Theorem \ref{thm: limiting rational curves},
we see that $Q$ can be deformed to a nodal rational curve in $|4A+2E_1 - E_2 - \ldots - E_r|$ on $\calX_\eta$.

Also it is easy to check that $\Gamma + P + Q$ has normal crossings on $Y_1$ and $Y_2$, respectively, and $p\not\in P\cup Q$. So its deformation on $\calX_\eta$ has normal crossings as well.
\end{proof}

Now we have produced nodal rational curves on K3 surfaces with Picard lattices \eqref{K3RATNOTESE007}
and \eqref{K3RATNOTESE019}. Theorem \ref{thm: nodal curves} follows more or less easily. \\

\begin{specialproof}[Proof of Theorem \ref{thm: nodal curves} when $\det(\Lambda)$ is even]
Let $Y$ be a general K3 surface with Picard lattice \eqref{K3RATNOTESE007} for $r=6$.
By Lemma \ref{lem: lattice1}
and \ref{lem: lattice3}, we can find a primitive lattice embedding
$\sigma: \Lambda \hookrightarrow \Pic(Y)$ such that $\sigma(L)$ is big and nef on $Y$.
Then there is a nodal rational curve $C\in |\sigma(L)|$ by Theorem \ref{K3NOTESTHM001}.

There is a smooth proper family $\pi: \calX\to \Spec \CC[[t]]$ of K3 surfaces such that $\calX_0 = Y$, $\calX_\eta$ has Picard lattice $\Lambda$ and $L$ extends to a divisor $\calL$ on $\calX$ with $\calL_0 = \sigma(L)$. Then $C$ can be deformed to a nodal rational curve
in $|\calL|$ on the generic fiber $\calX_\eta$ of $\calX$.
\end{specialproof}\\

\begin{specialproof}[Proof of Theorem \ref{thm: nodal curves} when $\det(\Lambda)$ is odd]
We are going to prove the theorem under the hypothesis A2 or A3.
Let $Y$ be a general K3 surface with Picard lattice \eqref{K3RATNOTESE019} for some $r\le 5$.
In both cases A2 and A3, it suffices to find a primitive lattice embedding
$\sigma: \Lambda \hookrightarrow \Pic(Y)$ such that $\sigma(L)$ is big and nef on $Y$ and
\begin{equation}\label{K3RATNOTESE047}
\left\{\begin{aligned}
\sigma(L) . A &\ge 3\\
\sigma(L) . E_5 &\le 2\hspace{12pt} \text{if } r = 5
\end{aligned}\right.
\end{equation}
where $A, E_1, E_2,\ldots, E_r$ are the effective generators of $\Pic(Y)$ with intersection matrix \eqref{K3RATNOTESE019}.

Suppose that $L$ satisfies A2. By Lemmas \ref{lem: lattice2} and \ref{lem: lattice3}, there is a primitive lattice embedding $\sigma: \Lambda \hookrightarrow \Pic(Y)$ for $r = 4$ such that $\sigma(L)$ is big and nef on $Y$.
In this case, we have $L = L_1 + L_2 + L_3$ such that $LL_i > 0$ and $L_i^2 > 0$
for $i=1,2,3$. Let us write
\begin{equation}\label{K3RATNOTESE023}
\sigma(L) = \sigma(L_1) + \sigma(L_2) + \sigma(L_3) = M_1 + M_2 + M_3
\end{equation}
for $M_i = \sigma(L_i)$. We claim that $M_i A \ge 1$ for all nef divisors $A\ne 0$ on $Y$.

Since $\sigma(L)$ is nef and $\sigma(L) . M_i = LL_i > 0$, $h^2(M_i) = h^0(-M_i) = 0$. Therefore, by Riemann-Roch,
$$
h^0(M_i) = h^1(M_i) + \frac{M_i^2}2 + 2 \ge \frac{L_i^2}2 + 2 > 2.
$$
Hence the linear system $|M_i|$ has a nonzero moving part. Let $\Gamma$ be an irreducible component of
the moving part of $|M_i|$. Then
$$
M_i A \ge \Gamma A \ge 0.
$$
If $\Gamma A > 0$, then $M_i A > 0$ follows. Otherwise, $\Gamma A = 0$. And since both $\Gamma$ and $A$ are nef, $A^2 = 0$ and
$\Gamma$ is numerically equivalent to $aA$ for some $a\in \QQ^+$ by the Hodge Index Theorem. This holds for all components $\Gamma$ of the moving part of $|M_i|$. So we have
$$
M_i \equiv a A + F
$$
where $F$ is the fixed part of $|M_i|$. If $F A > 0$, we again have $M_i A > 0$. Otherwise,
$FA = 0$; then
$$
F^2 = a^2 A^2 + 2aAF + F^2 = (aA + F)^2 = M_i^2 > 0.
$$
Since $F$ is effective and $F^2 > 0$, we again have $h^0(F) > 2$ by Riemann-Roch. This contradicts the fact that $F$ is the fixed part of $|M_i|$.

In conclusion, $M_i A \ge 1$ for all nef divisors $A\ne 0$ on $X$ and $i=1,2,3$. By \eqref{K3RATNOTESE023},
$\sigma(L) . A \ge 3$. This proves \eqref{K3RATNOTESE047} for case A2.

Suppose that $L$ satisfies A3. In this case, $L = L_1 + L_2$ such that $LL_i > 0$, $L_1^2 > 0$,
$L_2^2 = -2$, $L_1 \not\in 2\Lambda$, $L_1 - L_2 \not\in n\Lambda$ for all $n\in \ZZ$ and $n\ge 2$, and
\begin{equation}\label{K3RATNOTESE048}
L_1^2 + 2L_1L_2 \ge 18 \Leftrightarrow a + b \ge 9
\end{equation}
where we let $L_1^2 = 2a$ and $L_1 L_2 = b$.

Let us first assume that $\Lambda = \ZZ L_1 \oplus \ZZ L_2$.
In this case, we will use the numerical condition \eqref{K3RATNOTESE048} to explicitly construct a primitive embedding $\sigma: \Lambda\hookrightarrow \Pic(Y)$ for $r=5$ such that
$\sigma(L)$ is a big and nef divisor on $Y$ satisfying \eqref{K3RATNOTESE047}.

When $b \equiv 0\ (\text{mod } 3)$, we let
$$\left\{
\begin{aligned}
\sigma(L_1) &= \frac{a+9+\delta}3 A + 3 E_1 - \sum_{i=3}^{\delta+2} E_i,\ \text{for }
\delta = 3 + 3\left\lfloor\frac{a}3\right\rfloor - a
\\
\sigma(L_2) &= \frac{b}3 A - E_2
\end{aligned}
\right.
$$

When $b \equiv 1\ (\text{mod } 3)$, we let
$$\left\{
\begin{aligned}
\sigma(L_1) &= \frac{a+13+\delta}3 A + 3 E_1 - 2E_2 - \sum_{i=3}^{\delta+2} E_i,\ \text{for }
\delta = 2 + 3\left\lfloor\frac{a+1}3\right\rfloor - a
\\
\sigma(L_2) &= \frac{b-4}3 A + E_2
\end{aligned}
\right.
$$

When $b \equiv 2\ (\text{mod } 3)$, we let
$$\left\{
\begin{aligned}
\sigma(L_1) &= \frac{a+10+\delta}3 A + 3 E_1 - E_2 - \sum_{i=3}^{\delta+2} E_i,\ \text{for }
\delta = 2 + 3\left\lfloor\frac{a+1}3\right\rfloor - a
\\
\sigma(L_2) &= \frac{b-2}3 A + E_2
\end{aligned}
\right.
$$

It is easy to check that $\sigma(L) = \sigma(L_1 + L_2)$ is big and nef divisor on $Y$ satisfying \eqref{K3RATNOTESE047}. This settles the case $\Lambda = \ZZ L_1 \oplus \ZZ L_2$.

Assume now that
\begin{equation}\label{K3RATNOTESE049}
\Lambda \ne \ZZ L_1 \oplus \ZZ L_2.
\end{equation}
Let $\sigma: \Lambda \hookrightarrow \Pic(Y)$
be a primitive lattice embedding for $r = 4$ such that $\sigma(L)$ is big and nef on $Y$.
It suffices to prove $\sigma(L) . A \ge 3$.
We write
$$
\sigma(L) = \sigma(L_1 + L_2) = M_1 + M_2
$$
for $M_i = \sigma(L_i)$.
Since $L L_1 > 0$ and $L_1^2 > 0$, $M_1 A \ge 1$ by the same argument as before.
Since $LL_2 > 0$ and $L_2^2 = -2$, $M_2$ is effective by Riemann-Roch. So $M_2 A \ge 0$.
If $M_1 A + M_2 A \ge 3$, $\sigma(L) . A \ge 3$ and we are done. Otherwise, we have three cases:
\begin{itemize}
    \item $M_1 A = 1$ and $M_2 A = 0$;
    \item $M_1 A = 2$ and $M_2 A = 0$;
    \item $M_1 A = M_2 A = 1$.
\end{itemize}
We will show that none of these cases are possible.

Suppose that $M_1 A = 1$ and $M_2 A = 0$. Since $M_2$ is effective, $M_2^2 = -2$ and $M_2A = 0$, we necessarily have $M_2 = m A \pm E_j$
for some $2\le j \le 4$. And since $M_1 A = 1$,
it easy to see that $M_1$ and $M_2$ generate a primitive sublattice of
$\Pic(Y)$. Then $L_1$ and $L_2$ generate $\Lambda$, contradicting \eqref{K3RATNOTESE049}.

Suppose that $M_1 A = 2$ and $M_2 A = 0$. Again we have $M_2 = m A \pm E_j$. Then one of following must hold:
\begin{enumerate}
    \item $M_1$ and $M_2$ generate a primitive sublattice of
    $\Pic(Y)$,
    \item $M_1 = 2D$ for some $D\in \Pic(Y)$, or
    \item $M_1 - M_2 = 2D$ for some $D\in \Pic(Y)$.
\end{enumerate}
As pointed out above, the first case is equivalent to $\Lambda = \ZZ L_1 \oplus \ZZ L_2$,
contradicting \eqref{K3RATNOTESE049}. The second and third cases are equivalent to $L_1\in 2\Lambda$ and
$L_1 - L_2\in 2\Lambda$, respectively, both contradicting our hypotheses on $L_i$.

Suppose that $M_1 A = M_2 A = 1$. We have
\begin{enumerate}
    \item either $M_1$ and $M_2$ generate a primitive sublattice of
    $\Pic(Y)$ or
    \item $M_1 - M_2 = nD$ for some $D\in \Pic(Y)$, $n\in \ZZ$ and $n\ge 2$.
\end{enumerate}
Again the former contradicts \eqref{K3RATNOTESE049} and the latter is equivalent to
$L_1 - L_2\in n\Lambda$, contradicting our hypotheses on $L_i$.
This finishes the argument for case A3.

In conclusion, we can find a primitive embedding $\sigma:\Lambda\hookrightarrow \Pic(Y)$ such that
$\sigma(L)$ satisfies the hypotheses of Theorem \ref{K3NOTESTHM002}. So there is a nodal rational curve $C\in |\sigma(L)|$ on $Y$.

There is a smooth proper family $\pi: \calX\to \Spec \CC[[t]]$ of K3 surfaces such that $\calX_0 = Y$, $\calX_\eta$ has Picard lattice $\Lambda$ and $L$ extends to a divisor $\calL$ on $\calX$ with $\calL_0 = \sigma(L)$. Then $C$ can be deformed to a nodal rational curve
in $|\calL|$ on the generic fiber $\calX_\eta$ of $\calX$.
\end{specialproof}

\subsection{Higher rank lattices}

It is natural to expect the above techniques to apply to various lattice of higher rank, for the purposes of
\cite{regenerationinfinite} however, we will carry this out for the following two specific rank four lattices.

\begin{Theorem}\label{thm: curves on complex k3 picard rank 4}
Let $\Lambda$ be one of the following lattices of rank 4:
\begin{equation}\label{eq: curves on complex k3 picard rank 4 lattice 1}
\begin{bmatrix}
2 & -1 & -1 & -1 \\
-1 & -2 &  0 &  0 \\
-1 &  0 & -2 &  0 \\
-1 &  0 &  0 & -2 \\
\end{bmatrix}
\end{equation}
\begin{equation}\label{eq: curves on complex k3 picard rank 4 lattice 2}
\begin{bmatrix}
12 & -2 &  0 & 0 \\
-2 & -2 & -1 & 0 \\
0 & -1 & -2 & -1 \\
0 &  0 & -1 & -2 \\
\end{bmatrix}.
\end{equation}
Then for a general K3 surface $X$ with $\Pic(X) = \Lambda$, there is an integral rational (resp.\ geometric genus 1) curve in $|L|$ if
$L$ is a big and nef divisor $L$ on $X$ with the property that
\begin{equation}\label{eq: curves on complex k3 picard rank 4 00}
\begin{aligned}
L = L_1 + L_2 + L_3 \text{ for some }& L_i\in \Lambda \text{ satisfying that}\\
    &LL_i > 0 \text{ and } L_i^2 > 0 \text{ for } i=1,2,3.
\end{aligned}
\end{equation}
\end{Theorem}
\begin{proof}
We claim that there is a primitive embedding $\sigma: \Lambda\hookrightarrow\Sigma_r$ for $\Sigma_r$ given by
\eqref{K3RATNOTESE019} and $r\le 5$.

When $\Lambda$ is \eqref{eq: curves on complex k3 picard rank 4 lattice 1}, we let $r=4$ and
$$
\begin{aligned}
\sigma(B) &= 2A+E_1 \\
\sigma(C_1) &= -A+E_2 \\
\sigma(C_2) &= -A+E_3 \\
\sigma(C_3) &= -A+E_4
\end{aligned}
$$
where $\{B, C_1, C_2, C_3\}$ and $\{A, E_1,\ldots,E_r\}$ are the bases of $\Lambda$ and $\Sigma_r$, respectively,
with the corresponding intersection matrices.

When $\Lambda$ is \eqref{eq: curves on complex k3 picard rank 4 lattice 2}, we let $r=5$ and
$$
\begin{aligned}
\sigma(B) &= 12 A + 6 E_1 - 4E_2 - 3E_3 - 2E_4 - E_5 \\
\sigma(C_1) &= A - E_2 \\
\sigma(C_2) &= -E_1 \\
\sigma(C_3) &= A - E_3.
\end{aligned}
$$
This proves our claim.
So there exists a flat proper family $\pi: \calX\to \Spec\CC[[t]]$ of
K3 surfaces such that $\calX_0$ is a general K3 surface with Picard lattice $\Sigma_r$, $\calX_\eta$ is a K3 surface with Picard lattice $\Lambda$ and there is a divisor $\calL$ on $\calX$ with $\calL_0 = \sigma(L)$.

We may choose $\sigma$ such that $\calL_0 = \sigma(L)$ is big and nef on $\calX_0$. By the same argument as in the proof of
Theorem \ref{thm: nodal curves}, we can show that
$$
\sigma(L_i) . A \ge 1 \Rightarrow \sigma(L) . A \ge 3
\Leftrightarrow \calL_0 A \ge 3
$$
on $\calX_0$
by \eqref{eq: curves on complex k3 picard rank 4 00}.

When $\Lambda$ is \eqref{eq: curves on complex k3 picard rank 4 lattice 1},
$r=4$. Then the existence of nodal rational curves in
$|\calL_0|$ on $\calX_0$ is directly given by Theorem
\ref{K3NOTESTHM002}. Therefore, there are integral rational (resp.\ geometric genus 1) curves
in $|\calL|$ on $\calX_\eta$.

When $\Lambda$ is \eqref{eq: curves on complex k3 picard rank 4 lattice 2},
$r=5$. Theorem \ref{K3NOTESTHM002} only gives the existence of nodal rational curves in $|\calL_0|$ if $\calL_0$ additionally satisfies
$$
\min_{2\le i \le 5} \calL_0 E_i \le 2.
$$
So some extra work is needed. Suppose that
$$
\calL_0 = d A + m_1 E_1 - \sum_{i=2}^5 m_i E_i
$$
for some $d,m_i\in \ZZ$.
Without loss of generality, let us assume that
$m_2\ge m_3\ge m_4\ge m_5$. Since $\calL_0$ is nef and $\calL_0 A\ge 3$, we have
$$
d \ge 2m_1 \ge 4m_2\ge 4m_3\ge 4m_4\ge 4m_5\ge 0 \hspace{12pt}
\text{and} \hspace{12pt} m_1\ge 3.
$$
If $m_5 \le 1$, then there is a nodal rational curve in $|\calL_0|$ by Theorem \ref{K3NOTESTHM002} and we are done. Let us assume that $m_5\ge 2$.

We write
\begin{equation}\label{eq: curves on complex k3 picard rank 4 01}
\begin{aligned}
\calL_0 &= (d-4m_5+4) A + (m_1 - 2m_5+2) E_1 - \sum_{i=2}^5
(m_i - m_5 + 1) E_i\\
&\quad + (m_5-1) (4A + 2E_1 - \sum_{i=2}^5 E_i) = P + (m_5-1) F
\end{aligned}
\end{equation}
if $m_1-2m_5\ge 1$ and
\begin{equation}\label{eq: curves on complex k3 picard rank 4 02}
\begin{aligned}
\calL_0 &= (d-4m_5) A + m_5 (4A + 2E_1 - \sum_{i=2}^5 E_i)\\
&= (d-4m_5) A + m_5 F
\end{aligned}
\end{equation}
if $m_1 = 2m_5$, which implies $m_1 = 2m_2 = 2m_3=2m_4=2m_5$.

Suppose that $m_1-2m_5\ge 1$. That is, we have \eqref{eq: curves on complex k3 picard rank 4 01}. Since $P$ is big and
nef and $PA \ge 3$, there exists an integral rational (resp.\ geometric genus 1) curve $\Gamma\in |P|$ by Theorem \ref{K3NOTESTHM002}. There is also an integral nodal rational curve
$R\in |F|$ such that $\Gamma$ and $R$ meet transversely. As $F^2=0$, $R$ has a unique node $q$. Let
$$
\begin{tikzcd}
\widehat{\Gamma} \cup R_1 \cup R_2\cup \ldots \cup R_{m_5-1}\ar{r}{f} &\calX_0
\end{tikzcd}
$$
be a stable map given as follows:
\begin{itemize}
    \item $f: \widehat{\Gamma}\to \Gamma$ and $f: R_i\to R$ are the normalisations of $\Gamma$ and $R$, respectively, for $i=1,2,\ldots,m_5-1$;
    \item $\widehat{\Gamma}$ and $R_1$ meet at one point, $R_i$ and $R_{i+1}$ meet at one point for $i = 1,2,\ldots,m_5-2$ and there are no other intersections among $\Gamma$ and $R_i$;
    \item $f$ maps the point $\widehat{\Gamma}\cap R_1$
    to one of the intersections in $\Gamma\cap R$ and it is a local isomorphism at $\widehat{\Gamma}\cap R_1$ onto its image;
    \item $f$ maps the point $R_i\cap R_{i+1}$ to the node $q$ of $R$ and it is a local isomorphism at $R_i\cap R_{i+1}$ onto its image for $i=1,2,\ldots,m_5-2$.
\end{itemize}
By a local isomorphism at $\widehat{\Gamma}\cap R_1$ and $R_i\cap R_{i+1}$, we mean that $f$ maps an \'etale/analytic/formal neighborhood of the point
on the curve isomorphically onto its image.

It is clear that $f$ deforms in the expected dimension on $\calX_0$.
So it deforms to $\calX_\eta$. On the other hand, the divisor class $F$ does not deform in the family $\calX$ over $\Spec \CC[[t]]$ since $\calX_\eta$ is not elliptic. Therefore, $f$ extends to a family of stable maps to $\calX$ over $\Spec \CC[[t]]$, still denoted by $f: \mathscr{C}\to \calX$, such that
$\mathscr{C}_\eta$ is smooth and $f_* \mathscr{C}_\eta$ is an integral rational (resp.\ geometric genus 1) curve on $\calX_\eta$. We are done.

Suppose that $m_1=2m_5$. That is, we have \eqref{eq: curves on complex k3 picard rank 4 02}. There is a nodal rational curve $D\in |A|$ such that
$D$ and $R$ meet transversely at two points.
Clearly, $D$ has a unique node $p$. Let
$$
\begin{tikzcd}
D_1\cup D_2 \cup \ldots \cup D_{d-4m_5} \cup R_1 \cup R_2\cup \ldots \cup R_{m_5}\ar{r}{f} &\calX_0
\end{tikzcd}
$$
be a stable map given as follows:
\begin{itemize}
    \item $f: D_i\to D$ and $f: R_j\to R$ are the normalisations of $D$ and $R$, respectively, for $i=1,2,\ldots,d-4m_5$
    and $j=1,2,\ldots,m_5$;
    \item $D_i$ and $D_{i+1}$ meet at one point,
    $D_{d-4m_5}$ and $R_1$ meet at one point,
    $R_j$ and $R_{j+1}$ meet at one point for $1\le i \le d-4m_5-1$ and
    $1\le j\le m_5-1$, and
    there are no other intersections among $D_i$ and $R_j$;
    \item $f$ maps the point $D_{d-4m_5}\cap R_1$
    to one of the intersections in $D\cap R$ and it is a local isomorphism at $D_{d-4m_5}\cap R_1$ onto its image;
    \item $f$ maps the point $D_i\cap D_{i+1}$ to the node $p$ of $D$ and it is a local isomorphism at $D_i\cap D_{i+1}$ onto its image for $i=1,2,\ldots,d-4m_5-1$;
    \item $f$ maps the point $R_j\cap R_{j+1}$ to the node $q$ of $R$ and it is a local isomorphism at $R_j\cap R_{j+1}$ onto its image for $j=1,2,\ldots,m_5-1$.
\end{itemize}

Again $f$ deforms in the expected dimension on $\calX_0$. So it deforms to $\calX_\eta$. On the other hand, neither $A$ nor $F$ deforms in the family $\calX$ over $\Spec \CC[[t]]$ since $\calX_\eta$ is not elliptic. Therefore, $f$ extends to a family of stable maps to $\calX$ over $\Spec \CC[[t]]$, still denoted by $f: \mathscr{C}\to \calX$, such that
$\mathscr{C}_\eta$ is smooth and $f_* \mathscr{C}_\eta$ is an integral rational curve on $\calX_\eta$.

To see that there is also an integral geometric genus 1 curve in $|\calL|$ on $\calX_\eta$, we let
$s\in D\cap R$ be the intersection such that
$s\ne f(D_{d-4m_5} \cap R_1)$. Obviously, there are points
$s'\in D_{d-4m_5}$ and $s''\in R_1$ such that
$f(s') = f(s'') = s$. Therefore, $f(\mathscr{C}_\eta)$ has a singularity where it has two branches. Then it is well
known that $f(\mathscr{C}_\eta)$ can be deformed to an integral genus 1 curve on $\calX_\eta$ (see e.g.,\
\cite[Lemma 6.5]{regenerationinfinite}) so we are done.
\end{proof}

\bibliographystyle{alpha}
\newcommand{\etalchar}[1]{$^{#1}$}

\end{document}